\documentclass[A4paper,11pt,reqno]{amsart}

\usepackage[english]{babel}
\usepackage{amsmath}
\usepackage{amssymb} 
\usepackage{amsfonts}
\usepackage[all,color]{xy}
\usepackage{hyperref}
\usepackage{lscape}
\usepackage{graphicx}
\usepackage{float}

\usepackage{tikz}
\usetikzlibrary{intersections}
\tikzstyle{braid}=[thick]
\tikzset{arr/.style={circle,draw,inner sep=0.03cm}}
\tikzset{unit/.style={circle,fill,inner sep=0.05cm}}
\tikzset{empty/.style={inner sep=0pt, minimum size=0pt}}

 \numberwithin{equation}{section}
 \usepackage{color}

 \usepackage{hyperref}
 \usepackage{enumerate,xspace}
\usepackage{graphicx}

\renewcommand{\phi}{\varphi}
\renewcommand{\epsilon}{\varepsilon}

 \newtheorem{proposition}{Proposition}[section] 
 \newtheorem{lemma}[proposition]{Lemma}
 \newtheorem{theorem}[proposition]{Theorem}
 \newtheorem{corollary}[proposition]{Corollary}
 \theoremstyle{definition}
 \newtheorem{definition}[proposition]{Definition}
 \newtheorem{example}[proposition]{Example}
 
 \newtheorem{remark}[proposition]{Remark}
  \newtheorem{examples}[proposition]{Examples}

\newcommand{\ca}{\ensuremath{\mathcal A}\xspace}
\newcommand{\cb}{\ensuremath{\mathcal B}\xspace}
\newcommand{\cc}{\ensuremath{\mathcal C}\xspace}

\newcommand{\cs}{\ensuremath{\mathcal S}\xspace}

\newcommand{\cv}{\ensuremath{\mathcal V}\xspace}

\newcommand{\bba}{\ensuremath{\mathbb A}\xspace}

\newcommand{\myt}{t}

\def\ox{\otimes}
\def\x{\times}

\newcommand{\op}{\ensuremath{{}^{\textrm{op}}}\xspace}

\newcommand{\rev}{\ensuremath{{}^{\textrm{rev}}}\xspace}
\newcommand{\bCat}{\ensuremath{{\mathbb{C}\textnormal{at}}}\xspace}
\newcommand{\Vect}{\ensuremath{\mathbf{Vect}}\xspace}
\newcommand{\FProd}{\ensuremath{\mathbf{FProd}}\xspace}
\newcommand{\FProds}{\ensuremath{\mathrm{\FProd_s}}\xspace}

\newcommand{\RMod}{\textnormal{R-$\mathbf{Mod}$}\xspace}
\newcommand{\BMod}{\textnormal{B-$\mathbf{Mod}$}\xspace}
\newcommand{\Ab}{\textnormal{Ab}\xspace}

\newcommand{\Set}{\ensuremath{\mathbf{Set}}\xspace}
\newcommand{\Cat}{\ensuremath{\mathbf{Cat}}\xspace}

\newcommand{\TAlg}{\ensuremath{\mathbf{T}\textnormal{-}\mathbf{Alg}}\xspace}
\newcommand{\TMAlg}{\ensuremath{\mathbb{T}\textnormal{-Alg}}\xspace}

\newcommand{\inv}{\ensuremath{{}^{\textrm{inv}}}\xspace}

\def\ox{\otimes}
\def\x{\times}

\definecolor{lightgray}{rgb}{0.666666,0.666666,0.666666}
\definecolor{darkgreen}{rgb}{0,0.45,0}

\begin{document}

\title{Braided skew monoidal categories}

\author{John Bourke}
\address{Department of Mathematics and Statistics, Faculty of Science, Masaryk University, Brno, Czech Republic}
\email{bourkej@math.muni.cz}
\address{Department of Mathematics, Macquarie University, NSW 2109, Australia}
\email{john.d.bourke@mq.edu.au}

\author{Stephen Lack}
\address{Department of Mathematics, Macquarie University NSW 2109, 
Australia}
\email{steve.lack@mq.edu.au}


\begin{abstract}
We introduce the notion of a braiding on a skew monoidal category,
whose curious feature is that the defining isomorphisms involve three
objects rather than two.  Examples are shown to arise from 2-category
theory and from bialgebras.  In order to describe the 2-categorical
examples, we take a multicategorical approach.  We explain how certain braided skew monoidal structures
in the 2-categorical setting give rise to braided monoidal 
bicategories. For the 
bialgebraic examples, we show that, for a skew monoidal category arising from a
bialgebra, braidings on the skew monoidal category are in bijection
with cobraidings  (also known as coquasitriangular structures) on the
bialgebra. 
\end{abstract} 
\date\today
\maketitle

\section{Introduction}

A skew monoidal category is a category \cc equipped with a functor
$\cc^2\to \cc\colon (X,Y)\mapsto XY$,
an object $I \in\cc$, and natural transformations
\[ \xymatrix @R0pc { 
(XY)Z \ar[r]^{a} & X(YZ) \\
IX \ar[r]^{\ell} & X \\
X \ar[r]^{r} & XI
} \]
satisfying five coherence conditions \cite{Szlachanyi-skew}. When the
maps $a$, $\ell$, and $r$ are invertible, we recover the usual notion of monoidal
category. 

The generalisation allows for new examples.  For instance, if $B$ is a
bialgebra we obtain a new skew monoidal structure $\Vect[B]$ on 
the category $\Vect$ of vector spaces, with product $X
\star Y = X \otimes B \otimes Y$ and $I$ the ground field $K$. 
In this case the associativity map $a$ is defined using the
``fusion map'' of $B$, and is invertible just when the
bialgebra is Hopf; on the other hand the unit maps $\ell$ and $r$ are
never invertible unless $B=I$. 
More generally bialgebroids give rise to, and can by characterised by,
certain skew monoidal categories \cite{Szlachanyi-skew}.

Another class of examples \cite{bourko-skew}  arises if one attempts to study
2-categorical structures as strictly as possible.  For instance, there
is a skew monoidal structure on the 2-category $\FProds$ of
categories equipped a choice of finite products, and
functors which \emph{strictly} preserve them (not just in the usual up
to isomorphism sense). The tensor product $AB$ has the universal
property that maps $AB \to C$ correspond to functors $A \times B \to
C$ preserving products strictly in the first variable but up to
isomorphism in the second.  Although this example may seem
slightly bizarre, there is in fact good reason to study it. What
one really cares about is the 2-category $\FProd$ of categories with
finite products, finite-product-preserving functors (in the usual
up-to-isomorphism sense), and natural transformations.
But this is harder to work with --- for example, it 
has only bicategorical colimits.  In particular it has the structure of 
a monoidal closed bicategory, but the verification of this is technically rather
challenging.  The skew monoidal structure on $\FProds$ is much
easier to construct, and in fact --- as explained in Section 6.4.3 of
\cite{bourko-skew} --- contains the monoidal bicategory structure on $\FProd$ within it. 

Instead of the categories with finite products appearing in the previous
paragraph, one can do much
the same thing with structures such as symmetric monoidal categories,
or permutative categories, leading the way open to possible
applications to K-theory, using  \cite{Elmendorf2009Permutative}.

A natural question to ask is whether there exists a sensible notion of
\emph{braiding} for skew monoidal categories, generalising the
classical theory of braided monoidal categories
\cite{JoyalStreet-braided}.  A naive approach would be to
ask for an invertible natural transformation $$s\colon AB \to BA$$
interacting suitably with the skew monoidal structure.  However we
would like our notion of braiding to capture the example of $\FProds$
and, in that case, the objects $AB$ and $BA$ are not isomorphic.
Instead, what we find is that $(AB)C$ and $(AC)B$ both classify
functors preserving products strictly in $A$ and up to isomorphism in
$B$ and $C$, and so are isomorphic.

In the present paper we introduce a notion of braiding on a skew monoidal category which is given by an invertible natural transformation
$$s\colon (AB)C \to (AC)B$$
satisfying certain axioms.  Apart from capturing the above example and
others like it, the definition is justified in various ways.  For
example in Theorem~\ref{thm:bialgebras} we establish that braidings on
the skew monoidal category $\Vect[B]$ are in bijection
with \emph{cobraidings} (also known as coquasitriangular structures)
\cite{Kassel-book,Street-book}  on the bialgebra $B$.

Let us now give a brief outline of the paper.  In
Section~\ref{sect:monoidal} we define braidings and describe various
consequences of the axioms --- in particular, showing that if the
underlying skew monoidal structure notion is monoidal then our
definition restricts to the classical one.  In
Section~\ref{sect:closed} we introduce the notion, perhaps more
intuitive, of a braided skew closed category.  In this setting the
braiding is specified by an isomorphism $$[A,[B,C]] \to [B,[A,C]]$$
just as in the classical setting of symmetric closed categories.  Sections~\ref{sect:bialgebra} and~\ref{sect:2cat} are
driven by our two leading classes of examples.  Motivated by
bialgebras, in Section~\ref{sect:bialgebra} we study skew cowarpings
and monoidal comonads on monoidal categories.  The main result of
Section \ref{sect:bialgebra} is Theorem~\ref{thm:C^G-C[G]},
which asserts that, given a monoidal comonad $G$ on a monoidal category
$\cc$ satisfying a mild hypothesis, there is a bijection between
braidings on the monoidal category $\cc^{G}$ of coalgebras and
braidings on the  cowarped skew monoidal category $\cc[G]$.  This is
then specialised to the bialgebra setting  in
Theorem~\ref{thm:bialgebras}.  In Section~{\ref{sect:multicat}}
we introduce braided skew multicategories and show how to pass from
these, assuming a representability condition, to braided skew monoidal
categories. We use this in Section~\ref{sect:2cat} to exhibit  braidings on the 2-categorical examples such as $\FProds$.

\subsection*{Acknowledgements}
Both authors acknowledge with gratitude the support of an Australian Research Council
Discovery Grant DP160101519.

\section{Braided skew monoidal categories}
\label{sect:monoidal}

Let \cc be a skew monoidal category with structure maps $a\colon
(AB)C\to A(BC)$, $\ell\colon IA\to A$, and $r\colon A\to AI$.

\begin{remark}\label{rmk:duals} 
There is a variant of the notion of skew monoidal category in which
the directions of $a$, $\ell$, and $r$ are all reversed. We call this
a {\em right skew monoidal category}. (Our skew monoidal categories
are also called {\em left skew}.) If \cc is skew monoidal then
there are induced right skew monoidal structures on the opposite
category $\cc$, and also on $\cc$ with reverse multiplication; we call
the latter $\cc\rev$. On the other hand if we use the reverse
multiplication on $\cc\op$ we get another (left) skew monoidal
category, called $\cc\op\rev$.
\end{remark}

\begin{definition}\label{defn:braiding}
A {\em  braiding} on \cc consists of natural isomorphisms $s\colon (XA)B\to
(XB)A$ making the following diagrams commute:
\begin{equation}
  \label{eq:S2}
\xymatrix{
& ((XA)C)B \ar[r]^{s1} & ((XC)A)B \ar[dr]^{s} \\
((XA)B)C \ar[ur]^{s} \ar[dr]_{s1} &&& ((XC)B)A \\
& ((XB)A)C \ar[r]_{s} & ((XB)C)A \ar[ur]_{s1} }  
\end{equation}

\begin{equation}
  \label{eq:S3a}
\xymatrix{
  ((XA)B)C \ar[r]^{s1} \ar[d]_{a} & ((XB)A)C \ar[r]^{s} & ((XB)C)A
  \ar[d]^{a1} \\
(XA)(BC) \ar[rr]_{s} && (X(BC))A }
\end{equation}

\begin{equation}
  \label{eq:S3b}
  \xymatrix{
((XA)B)C \ar[r]^{s} \ar[d]_{a1} & ((XA)C)B \ar[r]^{s1} & ((XC)A)B
\ar[d]^{a} \\
(X(AB))C \ar[rr]_{s} && (XC)(AB) }
\end{equation}

\begin{equation}
  \label{eq:S*}
\xymatrix{
((XA)B)C \ar[r]^{a1} \ar[d]_{s} & (X(AB))C \ar[r]^{a} & X((AB)C) \ar[d]^{1s} \\
((XA)C)B \ar[r]_{a1} & (X(AC))B \ar[r]_{a} & X((AC)B)
}
\end{equation}
The braiding is a  {\em symmetry} if the diagram
\begin{equation}
\label{eq:S1}
\xymatrix{ 
& (XB)A \ar[dr]^{s} \\ 
(XA)B \ar[ur]^{s} \ar@{=}[rr] && (XA)B }  
\end{equation}
commutes, wherein the equality symbol represents the identity map.
\end{definition}

In later sections we shall study in detail the various examples of
braided skew monoidal categories described in the introduction;
here we content ourselves with giving a rather simple class of examples, including 
a symmetric skew monoidal structure on the category of pointed sets.

\begin{example}\label{ex:Msets}
  Let \cv be a monoidal category and $M=(M,m,i)$ a monoid in \cv; we write as
  if \cv were strict. Write $\cv^M$ for the category of left
  $M$-modules; these consist of an object $X\in\cv$ equipped with an
  associative and unital action $x\colon MX\to X$. The category
  $\cv^M$  has a skew monoidal structure with 
  product $(X,x)\ox (Y,y)=(XY,xY\colon MXY\to XY)$, unit $(M,m)$, and
  associativity inherited from $\cv$. The left unit map $(M,m)\ox
  (X,x)\to (X,x)$ is given by $x\colon MX\to X$ and the right unit map
  $(X,x)\to (M,m)\ox (X,x)$ by $Xi\colon X\to XM$. If $\cv$ has a
  braiding $c$, then $\cv^M$ becomes braided via the isomorphisms
  \[ \xymatrix @C1pc {
      (X,x)\ox(Y,y)\ox(Z,z) \ar@{=}[d] \ar[r]^{s} & (X,x)\ox(Y,y)\ox(Z,z)
      \ar@{=}[d] \\
      (XYZ,xYZ) \ar[r]^{1c} &     (XZY,xZY) } \]
  and this will be a symmetry if $c$ is one. 
\end{example}

\begin{example}\label{ex:pointed} 
In particular, we may take $\cv$ to be the category of sets,
with symmetric monoidal structure given by coproduct, and take $M$ to be  the terminal monoid.
Then $\cv^M$ is the category of pointed sets. The product of pointed sets $(X,x_0)$ and $(Y,y_0)$ is the
  disjoint union $X+Y$ with point $x_0\in X\subseteq X+Y$ and the unit
  object is the singleton pointed set.
\end{example}

 We shall see in the proof of Proposition~\ref{prop:monoidal}
below that \eqref{eq:S3a} and \eqref{eq:S3b} are analogues of the braid
equations and \eqref{eq:S2} of the Yang-Baxter equation, while
\eqref{eq:S*} like \eqref{eq:S2} is automatic in the classical setting
of braided monoidal categories. 

\begin{remark}\label{rmk:s-inverse}
Observe that \eqref{eq:S3b} for $s$ is precisely \eqref{eq:S3a} for
$s^{-1}$. On the other hand, \eqref{eq:S2} holds for $s^{-1}$ if and
only if it does so for $s$; and the same is true of
\eqref{eq:S*}. We write $\cc\inv$ for the skew monoidal category \cc
equipped with the natural isomorphism $s^{-1}$.
If $s$ is a symmetry, so that $s^{-1}=s$, then \eqref{eq:S3b} is equivalent to \eqref{eq:S3a}, and $\cc\inv=\cc$.
\end{remark}

There is no explicit compatibility requirement between the braiding and the left and right unit
maps, but see Propositions~\ref{prop:s.r} and~\ref{prop:slr} below.

\begin{proposition}\label{prop:monoidal}
  If \cc is a monoidal category, braidings and symmetries in the sense
  of Definition~\ref{defn:braiding} are in bijection with braidings and
  symmetries in the usual sense. 
\end{proposition}

\proof
There is a unique natural isomorphism $c\colon BC\to CB$ making the
diagram
\[ \xymatrix{
(IB)C \ar[r]^{\ell1} \ar[d]_{s} & BC \ar[d]^{c} \\
(IC)B \ar[r]_{\ell1} & CB } \]
commute; now use \eqref{eq:S*} with $A=I$ 
\[ \xymatrix{
&&& (XB)C \ar[dr]^{a} \\ 
(XB)C \ar[r]_{(r1)1} \ar[d]_{s} \ar@/^1pc/[urrr]^{1} & ((XI)B)C \ar[r]_{a1} \ar[d]_{s}
& (X(IB))C \ar[r]_{a} \ar[ur]^{(1\ell)1} & X((IB)C)
\ar[d]^{1s} \ar[r]_{1(\ell1)} & X(BC) \ar[d]^{1c} \\
(XC)B \ar[r]^{(r1)1} \ar@/_1pc/[drrr]_{1} & ((XI)C)B \ar[r]^{a1} &
(X(IC))B \ar[r]^{a} \ar[dr]_{(1\ell)1} & X((IC)B)
\ar[r]^{1(\ell1)} & X(CB) \\
&&& (XC)B \ar[ur]_{a} 
} \]
to deduce that $s$ necessarily has the
form 
\[ \xymatrix{
(XB)C \ar[r]^{a} & X(BC) \ar[r]^{1c} & X(CB) \ar[r]^{a^{-1}} & (XC)B.
} \]
Then \eqref{eq:S3a} and
\eqref{eq:S3b} are equivalent to the usual two axioms
\cite{JoyalStreet-braided} for a braiding,
and \eqref{eq:S2} is a consequence  by
\cite[Proposition~1.2]{JoyalStreet-braided}. \eqref{eq:S*} is automatic by
\[ \xymatrix{
  ((XA)B)C \ar[r]^{a1} \ar[d]^{a} \ar@/_3pc/[ddd]_{s} & (X(AB))C \ar[r]^{a} & X((AB)C) \ar[d]_{1a} \ar@/^3pc/[ddd]^{1s} \\
  (XA)(BC) \ar[rr]^{a} \ar[d]^{1c} && X(A(BC)) \ar[d]_{1(1c)} \\
  (XA)(CB) \ar[rr]^{a} \ar[d]^{a^{-1}} && X(A(CB)) \ar[d]_{1a^{-1}} \\
  ((XA)C)B \ar[r]_{a1} & (X(AC))B \ar[r]_{a} & X((AC)B) } \]
while finally \eqref{eq:S1} is clearly 
equivalent to the usual symmetry axiom for $c$. 
\endproof

Thus if $a$, $\ell$, and $r$ are all invertible, then we recover the
usual notion of braided or symmetric monoidal category. But in fact it
is enough just to suppose that $\ell$ is invertible: see
Proposition~\ref{prop:left-normal} below.

\subsection*{Consequences of the axioms}

Let $s$ be a braiding on the skew monoidal category \cc. 

\begin{lemma}\label{lemma:s.r}
If~\eqref{eq:S3a} holds then the composite
\[ \xymatrix @R0pc {
(WB)A \ar[r]^{r} & ((WB)A)I \ar[r]^{s} & ((WB)I)A \ar[r]^{a1} &
(W(BI))A } \]
is equal to $(1r)1$.
\end{lemma}

\proof
This holds by commutativity of 
\[ \xymatrix{
& (WB)A \ar[dr]^{r} \\
(WA)B \ar[r]_{r} \ar[ur]^{s} \ar[dr]_{1r} \ar[d]_{s} & ((WA)B)I \ar[r]_{s1}
\ar[d]^{a} &
((WB)A)I \ar[r]_{s} & ((WB)I)A \ar[d]^{a1} \\
(WB)A \ar@/_2pc/[rrr]_{(1r)1} & (WA)(BI) \ar[rr]^{s} && (W(BI))A } \]
and invertibility of $s$. \endproof

\begin{proposition}\label{prop:s.r}
  If \eqref{eq:S3a} holds then the diagram 
\[ \xymatrix{
WA \ar[r]^{r} \ar[dr]_{r1} & (WA)I \ar[d]^{s} \\
& (WI)A } \]
commutes.
\end{proposition}

\proof
Use the previous lemma in
\[ \xymatrix{
WA \ar[r]^r \ar[d]_{r1} & (WA)I \ar[r]^{s} \ar[d]_{(r1)1} & (WI)A
\ar[d]_{(r1)1}  \ar@/^3pc/[ddd]^{1} \\
(WI)A \ar[r]^{r} \ar[drr]^{(1r)1} \ar[ddrr]_{1} & ((WI)A)I \ar[r]^{s} & ((WI)I)A
\ar[d]_{a1} \\
&& (W(II))A \ar[d]_{(1\ell)1} \\
&& (WI)A.  } \]
\endproof

This easily implies:

\begin{proposition}\label{prop:slr}
  If \eqref{eq:S3a} holds then the diagram 
\begin{equation}
  \label{eq:S4}
\xymatrix{
(XA)I \ar[r]^{s} & (XI)A \ar[r]^{a} & X(IA) \ar[d]^{1\ell} \\
XA \ar[u]^{r} \ar@{=}[rr] && XA }
\end{equation}
commutes.
\end{proposition}

\begin{proposition}\label{prop:s.a.r1}
  If \eqref{eq:S3a} holds then the diagram 
\[ \xymatrix{
(XA)B \ar[r]^{r1} \ar[d]_{s} & ((XA)I)B \ar[r]^{a} & (XA)(IB) \ar[d]^{s} \\
(XB)A \ar[r]_{(r1)1} & ((XI)B)A \ar[r]_{a1} & (X(IB))A } \]
commutes.
\end{proposition}

\proof
\[ \xymatrix{
(XA)B \ar[r]^{r1} \ar[dr]_{(r1)1} \ar[dd]_{s} & ((XA)I)B \ar[r]^{a}
\ar[d]^{s1} &
(XA)(IB) \ar[dd]^{s} \\
& ((XI)A)B \ar[d]^{s}  \\
(XB)A \ar[r]_{(r1)1} & ((XI)B)A \ar[r]_{a1} & (X(IB))A } \]
\endproof

Dually, we have

\begin{proposition}\label{prop:s.r1} 
 If~\eqref{eq:S3b} holds then the diagrams 
\[ \xymatrix{
WA \ar[r]^{r1} \ar[dr]_{r} & (WI)A \ar[d]^{s} \\ & (WA)I } \]
\[ \xymatrix{
(XA)B \ar[r]^{(r1)1} \ar[d]_{s} & ((XI)A)B \ar[r]^{a1} & (X(IA))B \ar[d]^{s} \\
(XB)A \ar[r]_{r1} & ((XB)I)A \ar[r]_{a} & (XB)(IA) } \]
 commute. 
\end{proposition}

\proof
Apply  Propositions~\ref{prop:s.r} and~\ref{prop:s.a.r1} to $\cc\inv$.  
\endproof

Recall that a skew monoidal category is {\em left normal} when the
left unit maps $\ell\colon IA\to A$ are invertible.

\begin{proposition}\label{prop:left-normal}
If \cc is a braided skew monoidal category which is left normal, then
 \cc is in fact monoidal.
\end{proposition}

\proof
As in Proposition~\ref{prop:monoidal}, there is a unique natural isomorphism $c\colon AB\to
BA$ making the diagram
\[ \xymatrix{
(IA)B \ar[r]^{s} \ar[d]_{\ell1} & (IB)A \ar[d]^{\ell1} \\ AB
\ar[r]_{c} & BA } \]
commute, and then by Proposition~\ref{prop:slr}, the diagram
\[ \xymatrix{
IB \ar@/^2pc/[rrrr]^1 
\ar[r]_{r} \ar[d]_{\ell} & (IB)I \ar[r]_{s} \ar[d]^{\ell1} & (II)B
\ar[r]_a \ar[d]_{\ell1} & I(IB) \ar[dl]^{\ell} \ar[r]_{1\ell} & IB \ar[d]^{\ell}  \\
B \ar[r]_{r} & BI \ar[r]_{c} & IB \ar[rr]_{\ell} & & B  } \] 
commutes. Since $c\colon BI\to IB$ and the various instances of $\ell$
are invertible, it follows that $r\colon B\to BI$ is also invertible;
thus the skew monoidal category \cc is also right normal.

By \eqref{eq:S3a} and one of the skew monoidal category axioms, the
diagram
\[ \xymatrix{
((IA)B)C \ar[r]^{s1} \ar[d]_{a} & ((IB)A)C \ar[r]^{s} & ((IB)C)A
\ar[d]_{a1} \ar[dr]^{(\ell1)1} \\
(IA)(BC) \ar[rr]_{s} & & (I(BC))A \ar[r]_{\ell1} & (BC)A } \]
commutes, and so the left vertical is invertible. But now by
naturality the diagram 
\[ \xymatrix{
((IA)B)C \ar[r]^{(\ell1)1} \ar[d]_{a} & (AB)C \ar[d]^{a} \\
(IA)(BC) \ar[r]_{\ell1} & A(BC) } \]
commutes, and so the right vertical is invertible. This proves that
the skew monoidal category \cc is actually monoidal.
\endproof

 We close this section with something which is {\em not} a
consequence of the axioms. In a braided skew monoidal category the
diagram
\[ \xymatrix @C1pc {
    ((XI)A)B \ar[rr]^{s1} \ar[d]_{a1} && ((XA)I)B \ar[d]^{a} \\
    (X(IA))B \ar[dr]_{(1\ell)1} && (XA)(IB) \ar[dl]^{1\ell} \\
    & (XA)B } \]
need not commute. 
In particular, it does not commute in the braided skew monoidal category of pointed sets
described in Example~\ref{ex:pointed}.  Note,
however, that composing either composite with $(r1)1\colon (XA)B\to
((XI)A)B$ gives the identity.

\section{Braided skew closed categories}
\label{sect:closed}

Let \cc be a skew closed category in the sense of \cite{Street-SkewClosed} with structure maps 
\[ \xymatrix @R0pc { 
[B,C] \ar[r]^-{L} & [[A,B],[A,C]] \\
[I,A] \ar[r]^{i} & A \\
I \ar[r]^{j} & [A,A]. } \]

\begin{definition}\label{defn:braiding-closed}
A {\em  braiding} on \cc consists of natural
isomorphisms $$s^{\prime}\colon [B,[A,Y]] \to [A,[B,Y]]$$ making the
following diagrams commute.  

\begin{equation}\label{eq:bourkeS2}
 \xymatrix{
[C,[B,[A,Y]]] \ar[r]^{s'} \ar[d]_{[1,s']} & [B,[C,[A,Y]]]
\ar[r]^{[1,s']} & [B,[A,[C,Y]]] \ar[d]^{s'} \\
[C,[A,[B,Y]]] \ar[r]_{s'} & [A,[C,[B,Y]]] \ar[r]_{[1,s']} &
[A,[B,[C,Y]]] } 
\end{equation}

\begin{equation}
\label{eq:bourkeS3}
  \xymatrix{
[B,[A,Y]] \ar[rr]^{s'} \ar[d]_{L} && [A,[B,Y]] \ar[d]^{[1,L]} \\
[[C,B],[C,[A,Y]]] \ar[r]_{[1,s']} & [[C,B],[A,[C,Y]]] \ar[r]_{s'} &
[A,[[C,B],[C,Y]]] 
}
\end{equation}
\begin{equation}
\label{eq:bourkeS3b}
  \xymatrix{
[B,[A,Y]] \ar[rr]^{s'} \ar[d]_{[1,L]} && [A,[B,Y]] \ar[d]^{L} \\
[B,[[C,A],[C,Y]]]  \ar[r]_{s'} & [[C,A],[B,[C,Y]]] \ar[r]_{[1,s']} &
[[C,A],[C,[B,Y]]]
}
\end{equation}

\begin{equation}\label{eq:bourkeS*} \xymatrix{
    [X,Y] \ar[r]^-{L} \ar[d]_{L} & [[C,X],[C,Y]] \ar[r]^-{L} & [[B,[C,X]],[B,[C,Y]]] \ar[d]^{[1,s']} \\
    [[B,X],[B,Y]] \ar[r]_-{L} & [[C,[B,X]],[C,[B,Y]]] \ar[r]_-{[s',1]} & [[B,[C,X]],[C,[B,Y]]] }
  \end{equation}
  
The braiding is a  {\em symmetry} if the diagram 
\begin{equation}\label{eq:bourkeS1}
 \xymatrix{ [B,[A,Y]] \ar[r]^{s'} \ar[dr]_1 & [A,[B,Y]] \ar[d]^{s'} \\
&    [B,[A,Y]] } \end{equation}
commutes.

\end{definition}
\begin{remark}   
Condition \eqref{eq:bourkeS3b} holds for $s'$ just when
\eqref{eq:bourkeS3} does for the inverse of $s'$; thus in the
symmetric case \eqref{eq:bourkeS3b} is not needed.  
A definition of symmetric skew closed category was given in
\cite{bourko-skew} --- a skew closed category equipped with a natural
isomorphism $s^{\prime}$ satisfying \eqref{eq:bourkeS2},
\eqref{eq:bourkeS3}, \eqref{eq:bourkeS1}, and an axiom concerning the
unit $I$.  We prove in Corollary~\ref{cor:red} that the unit axiom is
redundant, and so our definition implies that of
\cite{bourko-skew}.  In fact the implication is strict: we
shall see in Remark~\ref{rmk:quasicommutative} that \eqref{eq:bourkeS*} does not follow from the
other axioms.
\end{remark}

\begin{remark}
  Just as in the classical case, a skew-closed category \cc admits an enrichment over itself, and the representable functors $[B,-]$ are \cc-enriched. The condition \eqref{eq:bourkeS*}  states that the isomorphisms
  \[ s'\colon [B,[C,-]]\to [C,[B,-]]\]
  are not just natural, but \cc-natural. 
\end{remark}

Suppose that \cc is a  skew monoidal category which is
\emph{closed}, by which we mean that
each functor $-\ox A\colon \cc\to\cc$ has a right adjoint $[A,-]$, so that there are natural
isomorphisms $\cc(XA,B)\cong \cc(X,[A,B])$. 
The skew monoidal structure gives rise to a skew closed structure \cite{Street-SkewClosed},
with maps 
\[ \xymatrix @R0pc { 
[B,C] \ar[r]^-{L} & [[A,B],[A,C]] \\
[I,A] \ar[r]^{i} & A \\
I \ar[r]^{j} & [A,A]. } \]
The associativity map $a$ determines a map $t\colon [AB,Y]\to
[A,[B,Y]]$, which may be constructed from $L$ as the composite
\[ \xymatrix{
[AB,C] \ar[r]^-{L} & [[B,AB],[B,C]] \ar[r]^-{[u,1]} & [A,[B,C]] } \]
where $u$ is the unit of the tensor-hom adjunction. Conversely, $L$ can be constructed from $t$ as the composite
\[ \xymatrix{
  [B,C] \ar[r]^-{[\epsilon,1]} & [[A,B]A,C] \ar[r]^{t} & [[A,B],[A,C]] } \]
where $\epsilon\colon [A,B]A\to B$ is the counit of the tensor-hom adjunction.

There is a bijection between natural isomorphisms 
\begin{equation}
s\colon (XA)B \to (XB)A
\end{equation}
and natural isomorphisms
\begin{equation}
s^{\prime}\colon [A,[B,Y]] \to [B,[A,Y]]
\end{equation}
as related by the commutative square
\[ \xymatrix{
\cc((XA)B,Y) \ar[d] \ar[rr]^{\cc(s,1)} && \cc((X B)A ,Y) \ar[d] \\
\cc(X,[A,[B,Y]]) \ar[rr]^{\cc(1,s^{\prime})} && \cc(X,[B,[A,Y]])
} \]
in which the vertical maps are composites of adjointness isomorphisms.  A useful way to think of this correspondence is to write $T_A\colon \cc\to\cc$ for the functor sending $X$ to $XA$, and $H_A$ for its right adjoint. Then, for given $A$ and $B$, the $s\colon (XA)B\to (XB)A$ can be seen as the components of a natural transformation $T_B T_A\to T_AT_B$. Since $T_A\dashv H_A$ and $T_B\dashv H_B$, we may compose adjunctions to obtain $T_BT_A\dashv H_A H_B$, and now $s'\colon H_AH_B\to H_BH_A$ is simply the mate of $s\colon T_AT_B\to T_BT_A$.

\begin{theorem}\label{thm:closed}
Let $\cc$ be closed skew monoidal.  The equations
\eqref{eq:S2}--\eqref{eq:S1} for a natural isomorphism $s\colon (XA)B \to
(XB)A$ correspond, in turn, to the equations
\eqref{eq:bourkeS2}--\eqref{eq:bourkeS1} for $s^{\prime}\colon [A,[B,Y]] \to
[B,[A,Y]]$.  In particular, there is a bijection between
braidings or symmetries on $\cc$ as a skew monoidal category
and those on $\cc$ as a skew closed category.
\end{theorem}

\begin{proof}
Routine calculation shows that \eqref{eq:S2} and \eqref{eq:bourkeS2}
are equivalent, and  likewise \eqref{eq:S1} and
\eqref{eq:bourkeS1}. The remaining cases require a little more work.
 First we establish the correspondence between \eqref{eq:S3a} and \eqref{eq:bourkeS3}.  The equation \eqref{eq:S3a} asserts the commutativity of
\[ \xymatrix{
  T_CT_BT_A \ar[r]^{1s} \ar[d]_{a1} & T_CT_AT_B \ar[r]^{s1} & T_AT_CT_B \ar[d]^{1a} \\
  T_{BC}T_A \ar[rr]_{s} && T_AT_{BC} } \]
and, on taking mates, we see that this is equivalent to
\[ \xymatrix{
  H_{BC}H_A \ar[rr]^{s'} \ar[d]_{t1} && H_A H_{BC} \ar[d]^{1t} \\
  H_BH_CH_A \ar[r]_{1s'} & H_BH_AH_C \ar[r]_{s'1} & H_AH_BH_C. } \]
There is a contravariant functor $P$ sending $X\in\cc$ to $H_XH_A$, and we may regard the domain $H_{BC}H_A$ of the above displayed equation as this contravariant functor applied to $X=T_CB$. Similarly there is a contravariant functor $Q$ sending $X$ to $H_AH_XH_C$, and the codomain of the displayed equation is $QB$. The equation asserts the equality of two natural maps $P(T_CB)\to Q(B)$.   
Taking mates once again, this time with respect to the adjunction
$T_C\dashv H_C$, and noting the contravariance of $P$ and $Q$, we
see that this is equivalent to an equation between two induced maps $P(B)\to Q(H_CB)$; specifically, to commutativity of
\[ \xymatrix{
  H_B H_A \ar[rr]^{s'}  \ar[d]_{L1} && H_AH_B \ar[d]^{1L} \\
  H_{H_CB}H_C H_A \ar[r]_{1s'} & H_{H_CB}H_AH_C \ar[r]_{s'1} & H_A H_{H_CB}H_C } \]
which is the displayed equation \eqref{eq:bourkeS3} of the proposition.

 Since \eqref{eq:S3b} for $s$ is \eqref{eq:S3a} for its inverse,
and similarly \eqref{eq:bourkeS3b} for $s'$ is \eqref{eq:bourkeS3} for
its inverse, the argument above shows that \eqref{eq:S3b} is
equivalent to \eqref{eq:bourkeS3b}. 

Finally we establish the correspondence between ~\eqref{eq:S*} and \eqref{eq:bourkeS*}.  First observe that a morphism $f\colon A\to B$ in \cc induces a natural transformation $T_f\colon T_A\to T_B$, whose component at an object $X$ is $1f\colon XA\to XB$; this in turn has a mate $H_f\colon H_B\to H_A$.
Then  we may express \eqref{eq:S*} as
\[ \xymatrix{
  T_CT_BT_A \ar[r]^{1a} \ar[d]_{s1} & T_CT_{AB} \ar[r]^{a} & T_{(AB)C} \ar[d]^{T_s} \\
  T_BT_CT_A \ar[r]_{1a} & T_BT_{AC} \ar[r]_{a} & T_{(AC)B} } \]
which, on taking mates, becomes
\[ \xymatrix{
  H_{(AC)B} \ar[r]^{t} \ar[d]_{H_s} & H_{AC} H_B \ar[r]^{t1} & H_AH_CH_B \ar[d]^{1s'} \\
  H_{(AB)C} \ar[r]_{t} & H_{AB} H_C \ar[r]_{t1} & H_A H_B H_C. } \]
We can regard this as an equality of maps $P(T_BT_cA)\to Q(A)$ for contravariant functors $P$ and $Q$, and so on taking mates as an equality of maps $P(A)\to Q(H_CH_BA)$; specifically, the equality of the upper composites, and hence also of the lower composites, in the following two diagrams.
\[ \xymatrix{
  H_A \ar[r]^{H_\epsilon} \ar[dr]^{H_\epsilon} \ar[ddr]_{L} & H_{(H_BA)B} \ar[r]^{H_{\epsilon1}} & H_{((H_CH_BA)C)B} \ar[dr]^{H_s} \\
  & H_{(H_CA)C} \ar[r]^{H_{\epsilon1}} \ar[d]^{t} & H_{((H_BH_CA)B)C} \ar[r]^{H_{(s'1)1}} \ar[d]^{t} & H_{((H_CH_BA)B)C} \ar[d]^{t} \\
  & H_{H_CA}H_C \ar[r]^{H_\epsilon1} \ar[dr]_{L1} & H_{(H_BH_CA)B}H_C \ar[r]^{H_{s'1}1} \ar[d]^{t1} & H_{(H_CH_BA)B}H_C \ar[d]^{t1} \\ 
  && H_{H_BH_CA}H_BH_C \ar[r]_{H_{s'}11} & H_{H_CH_BA}H_BH_C } \]
\[ \xymatrix{
  H_A \ar[r]^{H_{\epsilon}} \ar[dr]_{L} & H_{(H_BA)B} \ar[r]^{H_{\epsilon1}} \ar[d]^{t} & H_{((H_CH_BA)C)B} \ar[d]^{t} \\
  & H_{H_BA}H_B \ar[r]^{H_\epsilon1} \ar[dr]_{L} & H_{(H_CH_BA)C}H_B \ar[d]^{t1} \\
  && H_{H_CH_BA}H_CH_B \ar[r]^{1s'} & H_{H_CH_BA}H_BH_C } \]
The equality of the lower composites is the condition \eqref{eq:bourkeS*} stated in the proposition.
\end{proof}

We conclude the section by showing that the unit axiom (S4) of \cite{bourko-skew} is redundant.

\begin{proposition}
If \eqref{eq:bourkeS3} commutes, then so too does 
\begin{equation}
  \label{eq:s'r}
  \xymatrix{
[I,[B,C]] \ar[r]^{s'} \ar[dr]_{i} & [B,[I,C]] \ar[d]^{[1,i]} \\ &
[B,C]. } 
\end{equation}
\end{proposition}

\proof
By \eqref{eq:bourkeS3}, the large rectangular region in 
\[ \xymatrix @C1pc {
[A,[B,C]] \ar[d]_{L} \ar[rr]_{s'} \ar@/^2pc/[rrr]^{[i,1]} &&
[B,[A,C]] \ar[d]_{[1,L]}  \ar[ddr]^{[1,[i,1]]} & [[I,A],[B,C]]
\ar[dd]^{s'} \\
[[I,A],[I,[B,C]]] \ar[r]_{[1,s']} & [[I,A],[B,[I,C]]] \ar[r]^{s'}
\ar[dr]_{[1,[1,i]]} &
[B,[[I,A],[I,C]]] \ar[dr]_{[1,[1,i]]} \\
&& [[I,A],[B,C]] \ar[r]_{s'} & [B,[[I,A],C]] } \]
commutes, while the other two quadrilaterals commute by naturality of
$s'$, and the triangular region by one of the skew closed category
axioms. Thus the exterior commutes. Cancel the isomorphism $s'$ at the
end of each composite and set $A=I$ to deduce commutativity of the
upper region of the diagram 
\[ \xymatrix{
[I,[B,C]] \ar[d]^{L} \ar[rr]^{1} \ar@/_4pc/[ddd]_1 && [I,[B,C]]  \ar[d]_{[i,1]}
\ar@/^3pc/[dd]^{1} \\
[[I,I],[I,[B,C]]] \ar[r]^{[1,s']} \ar[d]^{[j,1]} & [[I,I],[B,[I,C]]]
\ar[d]_{[j,1]} \ar[r]^{[1,[1,i]]} & [[I,I],[B,C]] \ar[d]_{ [j,1]} \\
[I,[I,[B,C]]] \ar[r]_{[1,s']} \ar[d]^{i} & [I,[B,[I,C]]]
\ar[r]_{[1,[1,i]]}  & [I,[B,C]] \ar[d]_{i} \\
[I,[B,C]] \ar[r]_{s'} & [B,[I,C]] \ar[r]_{[1,i]} & [B,C] }
\]
in which the central regions commute by functoriality of the internal
hom, the lower region by naturality of $i$,  and the left and right
regions by skew closed category axioms. 
\endproof

\begin{corollary}\label{cor:red}
  If \eqref{eq:bourkeS3} commutes then axiom (S4) of
  \cite{bourko-skew} holds; that is, the composite 
\[ \xymatrix{
[B,C] \ar[r]^-{L} & [[B,B],[B,C]] \ar[r]^-{[j,1]} & [I,[B,C]]
\ar[r]^{s'} & [B,[I,C]] \ar[r]^-{[1,i]} & [B,C] } \]
is the identity. In particular, axiom (S4) of \cite{bourko-skew} is
redundant.  
\end{corollary}

\proof
Use \eqref{eq:bourkeS3} to replace the last two factors in the
displayed composite by $i\colon [I,[B,C]]\to [B,C]$, then use one of
the skew closed category axioms to deduce that the resulting composite
is the identity. 

Since \eqref{eq:bourkeS3} is (S3) of \cite{bourko-skew}, it follows
that (S4) is redundant as claimed.
\endproof

\section{Braided cowarpings and bialgebras}
\label{sect:bialgebra}

For this section, we suppose that \cc is in fact a monoidal category,
and often write as if it were strict, and write $X.Y$ for the product
of $X$ and $Y$. Some aspects would work more
generally for a skew monoidal category. 

\subsection{Skew cowarpings}

A {\em skew cowarping} on \cc is a skew warping \cite{skew} on
$\cc\op\rev$. Explicitly, this involves the following data 
\begin{itemize}
\item a functor $Q\colon \cc\to\cc$
\item an object $K\in\cc$
\item maps $v\colon QX.QY\to Q(X.QY)$
\item $v_0\colon I\to QK$
\item $k\colon K.QX\to X$
\end{itemize}
 subject to five axioms.  The ``cowarped'' tensor product is given by $X*Y=X.QY$ with unit $K$. 
The structure maps are
\[ \xymatrix @R0pc { 
(X*Y)*Z \ar@{=}[r] &X.QY.QZ \ar[r]^{1.v} & X.Q(Y.QZ) \ar@{=}[r] &
X*(Y*Z) \\
K*X \ar@{=}[r] & K.QX \ar[r]^{k} & X \\
& X \ar[r]^{1.v_0} & X.QK \ar@{=}[r] & X*K } \]
and this defines a skew monoidal category $\cc[Q]$. 

\begin{proposition}\label{prop:braided-cowarping}
Let $\cc$ be a monoidal category, $Q$ a skew cowarping, and $$y\colon QX.QY\to QY.QX$$ a natural isomorphism.
The induced maps 
\[ \xymatrix{
(X*Y)*Z \ar@{=}[r] & X.QY.QZ \ar[r]^{1.y} & X.QZ.QY \ar@{=}[r] &
(X*Z)*Y } \]
equip $\cc[Q]$ with the structure of a braided skew monoidal category if and only if the following diagrams commute.
\begin{equation}
  \label{eq:1}
  \xymatrix{
QX.QY.QZ \ar[r]^{y.1} \ar[d]_{1.y} & QY.QX.QZ \ar[r]^{1.y} & QY.QZ.QX
\ar[d]^{y.1} \\
QX.QZ.QY \ar[r]_{y.1} & QZ.QX.QY \ar[r]_{1.y} & QZ.QY.QX }
\end{equation}
\begin{equation}
  \label{eq:1a}
  \xymatrix{ QX.QY.QZ \ar[d]_{1.v} \ar[r]^{y.1} & QY.QX.QZ
    \ar[r]^{1.y}  & QY.QZ.QX    \ar[d]^{v.1} \\
QX.Q(Y.QZ) \ar[rr]_{y} && Q(Y.QZ).QX }
\end{equation}
\begin{equation}
  \label{eq:1b}
  \xymatrix{ QX.QY.QZ \ar[r]^{1.y} \ar[d]_{v.1} & QX.QZ.QY
    \ar[r]^{y.1} & QZ.QX.QY \ar[d]^{1.v} \\
Q(X.QY).QZ \ar[rr]_{y} & & QZ.Q(X.QY) }
\end{equation}
\begin{equation}
  \label{eq:funny}
  \xymatrix{
QX.QY.QZ \ar[r]^{v.1} \ar[d]_{1.y} & Q(X.QY).QZ \ar[r]^{v} &
Q(X.QY.QZ) \ar[d]^{Q(1.y)} \\
QX.QZ.QY \ar[r]_{v.1} & Q(X.QZ).QY \ar[r]_{v} & Q(X.QZ.QY) }
\end{equation}
\end{proposition}
Once again, \eqref{eq:1b} is just \eqref{eq:1a} for $y^{-1}$.

\begin{proof}
It is straightforward to see that the above four equations imply, in turn, the four equations \eqref{eq:S2} to \eqref{eq:S*} for a braiding on $\cc[Q]$.  In the opposite direction one obtains the above four equations above by taking the first variable in \eqref{eq:S2} to \eqref{eq:S*} to be $I$.

\end{proof}
We shall refer to a natural isomorphism $y \colon QX.QY \to QY.QX$  satisfying the above four axioms as a \emph{braiding} on the skew cowarping $Q$.  

\begin{remark}\label{rmk:Igenerator}
Many monoidal categories \cc have the following property: for any two
objects $X$ and $A$, the maps 
\[ \xymatrix{IA \ar[r]^{x1} & XA,} \]
where $x\colon I\to X$, are jointly epimorphic. This means that two
morphisms $f,g\colon XA\to B$ are equal provided that their composites
with $x1$ are equal for any $x$. In particular this is true if \cc is right-closed and $\cc(I,-)$ is faithful. Examples include the categories of
sets, of $R$-modules for a commutative ring $R$, or of topological spaces; non-examples include the category of categories and the category of chain complexes. In this context we can strengthen the preceding result.
\end{remark}
\begin{proposition}\label{prop:special}
Let $\cc$ be a monoidal category having the property that:
\begin{itemize}
\item 
for any two objects $X$ and $A$ the maps $x1\colon IA \to XA$, where $x\colon I\to X$, are jointly epimorphic.  
\end{itemize}
Let $Q$ be a skew cowarping on $\cc$.  The construction of
Proposition~\ref{prop:braided-cowarping} defines a bijection between braidings on $\cc[Q]$ and braidings on $Q$.
\end{proposition}
\begin{proof}  
Without the assumption we may obtain a braiding on $Q$ from a braiding $s$ on $\cc[Q]$ using $s\colon (I*Y)*Z\to (I*Z)*Y$.
Applying this to a braiding on $\cc[Q]$ arising, as in
Proposition~\ref{prop:braided-cowarping}, from one on $Q$ returns the
original braiding on $Q$.
But to see that every braiding on $\cc[Q]$ arises in this way, we  rely on the assumption.
\end{proof}

\subsection{Monoidal comonads}

A monoidal comonad $G=(G,\delta,\epsilon, G_2,G_0)$ on \cc determines a skew cowarping with $QX=GX$
and $K=I$, and with $v$ given by 
\[ \xymatrix{ GX.GY \ar[r]^-{1.\delta} & GX.G^2Y \ar[r]^{G_2} & G(X.GY)
  } \]
with
$v_0$ and $k$ defined using $G_0\colon GI\to I$ and $\epsilon$ respectively; 
conversely, any skew cowarping with $K=I$ arises in this way from
a monoidal comonad \cite[Proposition~3.5]{skew}.  

By a \emph{braiding} on the monoidal comonad $G$ we simply mean a braiding on the associated skew cowarping.

Given a monoidal comonad $G$, in addition to the cowarped skew
monoidal category $\cc[G]$, we can form the lifted monoidal structure on the
Eilenberg-Moore category $\cc^G$ of coalgebras.

\begin{theorem}\label{thm:eilenbergMoore}
For a monoidal comonad $G$ on a monoidal category \cc, there is a
bijection between braidings on the monoidal category $\cc^G$ and
braidings on $G$.
\end{theorem}

\proof 
First suppose that $c$ is a braiding on $\cc^G$.  In particular, for
any cofree algebras $GX$ and $GY$ there is an isomorphism $c\colon
GX.GY\to GY.GX$ in $\cc^G$, and this is natural in $X$ and $Y$. 
By \cite[Proposition~ 2.1]{JoyalStreet-braided} the Yang-Baxter
equation \eqref{eq:1} holds. Now \eqref{eq:1a} holds by commutativity of the diagram
\[ \xymatrix{ 
GX.GY.GZ \ar[r]^{c.1} \ar[d]_{1.v} \ar@/_1pc/[rr]_{c_{GX,GY.GZ}} & GY.GX.GZ \ar[r]^{1.c}
 & GY.GZ.GX \ar[d]^{v.1} \\
GX.G(Y.GZ) \ar[rr]_{c} && G(Y.GZ).GX } \]
in which the upper region commutes by one of the braid axioms for
$\cc^G$, and the lower one by naturality of the braiding with respect
to the morphism $v\colon GY.GZ\to G(Y.GZ)$ in $\cc^G$. Equation
\eqref{eq:1b} holds by a similar, dual, argument. To see that
\eqref{eq:funny} holds, first observe that the horizontal composites
have the form 
\[ \xymatrix{
GX.GY.GZ \ar[r]^{1.\delta.1} \ar[dr]_{1.\delta.\delta} & GX.G^2Y.GZ \ar[r]^{G_2.1}
\ar[d]^{1.1.\delta} & G(X.GY).GZ \ar[d]^{1.\delta} \\
& GX.G^2Y.G^2Z \ar[r]^{G_2.1} \ar[d]_{1.G_2} & G(X.GY).G^2Z \ar[d]^{G_2} \\
& GX.G(GY.GZ) \ar[r]_{G_2} & G(X.GY.GZ)  } \]
and now \eqref{eq:funny} takes the form 
\[ \xymatrix{
GX.GY.GZ \ar[r]^{1.\delta.\delta} \ar[d]_{1.c} & GX.G^2Y.G^2Z \ar[r]^{1.G_2} &
GX.G(GY.GZ) \ar[r]^{G_2} \ar[d]^{1.Gc} & G(X.GY.GZ) \ar[d]^{G(1.c)} \\
GX.GZ.GY \ar[r]_{1.\delta.\delta} & GX.G^2Z.G^2Y \ar[r]_{1.G_2} & 
GX.G(GZ.GY) \ar[r]_{G_2} & G(X.GZ.GY) } \]
where the left region commutes because $c$ is a $G$-coalgebra
homomorphism, and the right region by naturality of $G_2$.

Suppose conversely that $y\colon GX.GY\to GY.GX$ is a braiding on $G$.
First take $X=I$ in \eqref{eq:funny}, to deduce commutativity of 
\[ \xymatrix{
GY.GZ \ar[r]_{v_0.1.1} \ar[d]_{y} \ar@/^1pc/[rr]^{\delta.1} & GI.GY.GZ \ar[r]_{v.1} \ar[d]_{1.y}
& G^2Y.GZ \ar[r]_{v} & G(GY.GZ) \ar[d]^{Gy} \\
GZ.GY \ar@/_1pc/[rr]_{\delta.1} \ar[r]^{v_0.1.1} & GI.GZ.GY \ar[r]^{v.1} & G^2Z.GY \ar[r]^{v} &
G(GZ.GY) } \]
in which the horizontal composites are the coalgebra structure maps;
thus $y$ is a coalgebra homomorphism. 
 
Then
\eqref{eq:1a} and \eqref{eq:1b} imply \eqref{eq:S3a} and
\eqref{eq:S3b}, and so by Propositions~\ref{prop:s.r} and
\ref{prop:s.r1}, the diagrams
\[ \xymatrix{
GX \ar[r]^{1.G_0} \ar[dr]_{G_0.1} & GX.GI \ar[d]^{y} & GX
\ar[r]^{G_0.1} \ar[dr]_{1.G_0} & GI.GX \ar[d]^{y} \\
& GI.GX && GX.GI 
} \] 
commute. Then 
\[ \xymatrix{
GX.GZ \ar[d]^{1.G_0.1}  \ar@/_3pc/[ddd]_{1.\delta} \ar[dr]^{G_0.1.1}
\ar[rr]^{y} && GZ.GX \ar[d]_{G_0.1.1}  \ar@/^3pc/[ddd]^{\delta.1}\\
GX.GI.GZ \ar[r]^{y.1} \ar[d]^{1.1.\delta} & GI.GX.GZ \ar[r]^{1.y} &
GI.GZ.GX \ar[d]_{1.\delta.1} \\
GX.GI.G^2Z \ar[d]^{1.G_2} && GI.G^2Z.GX \ar[d]_{G_2.1} \\
GX.G^2Z \ar[rr]_{y} && G^2Z.GX 
} \] 
commutes by \eqref{eq:1a}, and similarly 
\begin{equation}\label{eq:y.delta1}  
\xymatrix{ 
GX.GZ \ar[rr]^{y} \ar[d]_{\delta.1} && GZ.GX \ar[d]^{1.\delta} \\
G^2X.GZ \ar[rr]_{y} && GZ.G^2X } 
\end{equation}
commutes by \eqref{eq:1b}. Combining these, we see that 
\begin{equation}\label{eq:y.delta} 
\xymatrix{
GX.GZ \ar[rr]^{y} \ar[d]_{\delta.\delta} & & GZ.GX
\ar[d]^{\delta.\delta} \\
G^2X.G^2Z \ar[rr]_{y} && G^2Z.G^2X } 
\end{equation}
commutes. 

Let $(A,\alpha)$ and $(B,\beta)$ be $G$-coalgebras. The rows of 
\[ \xymatrix{
A.B \ar[r]^-{\alpha.\beta} \ar@{.>}[d]_c & GA.GB \ar[d]^{y} \ar@<1ex>[r]^{G\alpha.G\beta}
\ar@<-1ex>[r]_{\delta.\delta} & G^2A.G^2B \ar[d]^{y}  \\
B.A \ar[r]^-{\beta.\alpha} & GB.GA \ar@<1ex>[r]^{G\beta.G\alpha}
\ar@<-1ex>[r]_{\delta.\delta} & G^2B.G^2A
} \]
are split equalizers in \cc and so are equalizers in $\cc^G$. The
solid vertical $y$s commute with the rows by naturality of $y$ and
commutativity of \eqref{eq:y.delta}, thus there is a unique induced
invertible $c\colon A.B\to B.A$ making the left square commute.

It follows from \eqref{eq:y.delta}  that for cofree coalgebras $GX$
and $GY$, the map $c\colon GX.GY\to GY.GX$ is just
$y$. The braid axioms will hold for all coalgebras if and only if they
hold for cofree coalgebras. One of these holds by 
\begin{align*}
\xymatrix{
GX.GY.GZ \ar[r]^{\delta.1.\delta} \ar[d]_{c_{GX.GY,GZ}}
\ar[dr]^{\delta.\delta.\delta} & G^2X.GY.G^2Z \ar[dr]^{v.1} \\
GZ.GX.GY \ar[dr]_{\delta.\delta.\delta} & G^2X.G^2Y.G^2Z
\ar[r]^{G_2.1} \ar[d]^{c_{G^2X.G^2Y,G^2Z}} & G(GX.GY).G^2Z \ar[d]^{c} \\
& G^2Z.G^2X.G^2Y \ar[r]_{1.G_2} & G^2Z.G(GX.GY) 
}  
\\
\xymatrix{
GX.GY.GZ \ar[d]_{1.c} \ar[r]^{\delta.1.\delta} & G^2X.GY.G^2Z
\ar[dr]^{v.1} 
\ar[d]_{1.c}  \\
GX.GZ.GY \ar[d]_{c.1} & G^2.G^2Z.GY \ar[d]_{c.1} & G(GX.GY).G^2Z
\ar[dd]^{c} \\
GZ.GX.GY \ar[dr]_{\delta.\delta.\delta} \ar[r]_{\delta.\delta.1} & G^2Z.G^2X.GY \ar[dr]_{1.v} \\
& G^2Z.G^2X.G^2Y \ar[r]_{1.G_2}
& G^2Z.G(GX.GY)  
}   
\end{align*}
and the other is similar.
\endproof

Combining the above result with Proposition~\ref{prop:special} we obtain

\begin{theorem}\label{thm:C^G-C[G]}
Let  $G$ be a monoidal comonad on a monoidal category
$\cc$ having the property of Remark~\ref{rmk:Igenerator}: 
\begin{itemize}
\item 
for any two objects $X$ and $A$ the maps $x1\colon  IA \to XA$, where $x\colon I\to X$, are jointly epimorphic.  
\end{itemize}
Then there is a bijection between
\begin{enumerate}
\item braidings on $G$,
\item braidings on the monoidal category $\cc^G$ of coalgebras, and
\item braidings on the cowarped skew monoidal category $\cc[G]$.
\end{enumerate}
\end{theorem}

\subsection{The case of bialgebras}
Let \cv be a symmetric monoidal category, and $B$ a bialgebra in
\cv. The coalgebra structure of $B$ induces a comonad $G$ on \cv given
by tensoring on the left with $B$; the algebra structure comprising $\mu\colon
BB\to B$ and $\eta\colon I\to B$ makes this
into a monoidal comonad with structure maps
\[ \xymatrix @R0pc {
{} \llap{$GX.GY=$}   BXBY \ar[r]^{1c1} & BBXY \ar[r]^{\mu11} & BXY \rlap{$=G(XY)$} \\
    I \ar[rr]^{\eta1} && BI } \]
where $c\colon XB\to BX$ is the symmetry isomorphism.

In this setting we write $\cv[B]$ for the cowarped skew monoidal
structure which has tensor product $X \star Y = XBY$ and unit $I$.
The associator and unit
maps for $\cv[B]$ are given in the string diagrams below

\centerline{
\begin{tikzpicture}[scale=1]  
  \path (4,0) node {}
  (0,3) node[name=Xu] {$X$}
  (0.6,3) node[name=B1u] {$B$}
  (1.2,3) node[name=Yu] {$Y$}
  (1.8,3) node[name=B2u] {$B$}
  (2.4,3) node[name=Zu] {$Z$}
  (0,0) node[name=Xd] {$X$}
  (0.6,0) node[name=B1d] {$B$}
  (1.2,0) node[name=Yd] {$Y$}
  (1.8,0) node[name=B2d] {$B$}
  (2.4,0) node[name=Zd] {$Z$}
  (0.6,0.8) node[style=empty, name=m] {}
  (1.8,2.2) node[style=empty, name=d] {};
  \draw[braid] (Xu) to (Xd); 
  \draw[braid, name path=Y] (Yu) to (Yd); 
  \draw[braid] (Zu) to (Zd); 
  \draw[braid] (m) to (B1d);
  \draw[braid] (B2u) to (d);
  \draw[braid] (B1u) to [out=270, in=135] (m);
  \draw[braid] (d) to [out=315, in=90] (B2d);
  \path[braid, name path=dm] (d) to [out=225, in=45] (m); 
  \fill[white,name intersections={of=Y and dm}] (intersection-1) circle(0.1); 
  \draw[braid] (d) to [out=225, in=45] (m); 
\end{tikzpicture}
\begin{tikzpicture}
  \path (2,0) node {}
    (0.6,3) node[name=Xu] {$X$}
    (0,1.5) node[name=eta]  {}
  (0.6,0) node[name=Xd] {$X$}
  (0,3) node[name=Bu] {$B$} ;
  \draw[braid] (Xu) to (Xd);
  \fill (eta) circle(0.1) ;
  \draw[braid] (eta) to (Bu);
\end{tikzpicture}
\begin{tikzpicture}
  \path (2,0) node {}
    (0,3) node[name=Xu] {$X$}
    (0.6,1.5) node[name=eta]  {}
  (0,0) node[name=Xd] {$X$}
  (0.6,0) node[name=Bd] {$B$} ;
  \draw[braid] (Xu) to (Xd);
  \fill (eta) circle(0.1) ;
  \draw[braid] (eta) to (Bd);
\end{tikzpicture}
}
\noindent when read moving down the page, with the multiplication
represented by the merging of two strings, the comultiplication by the
splitting of two strings, and the unit and counit by the appearance or
disappearance of a string at a dot. 

If $\cv$ is closed then the isomorphisms $\cv(XBY,Z) \cong \cv(X,[BY,Z])$ ensure
that $\cv[B]$ is too.  This is the case, for instance, if $\cv=\RMod$
for a commutative ring $R$. 

\begin{remark} \label{rmk:quasicommutative}
There is an evident natural isomorphism $GX.GY\cong GY.GX$ with
components the symmetry isomorphisms $ c\colon (BX)(BY)\cong (BY)(BX)$.
The diagrams~\eqref{eq:1}, \eqref{eq:1a}, and \eqref{eq:1b} always
commute, as does the symmetry axiom, but \eqref{eq:funny} commutes if
and only if the algebra $B$ is commutative.  For the warped skew
monoidal category $\cv[B]$ the natural isomorphisms $1.c\colon X.BY.BZ \cong X.BY.BZ$ of Proposition~\ref{prop:braided-cowarping} satisfy all of the equations for a braided skew monoidal category except for \eqref{eq:S*}, which commutes just when the algebra $B$ is commutative.

Of course there are many bialgebras whose underlying algebra is not commutative, for instance the group ring of $\mathbb Z[G]$ of a non-abelian group $G$.  Therefore the cowarped skew monoidal structure $\Ab[\mathbb Z[G]]$ exhibits the independence of \eqref{eq:S*} from the other axioms for a braiding.  Since this skew monoidal structure is closed, it follows from Theorem~\ref{thm:closed} that \eqref{eq:bourkeS*} is independent of the other axioms for a  braided/symmetric skew closed structure.
\end{remark}

\begin{theorem}\label{thm:bialgebras}
Let $R$ be a commutative ring and $B$ an $R$-bialgebra.  There
are bijections between
\begin{enumerate}
\item Cobraidings (coquasitriangular structures) on the bialgebra $B$;
\item Braidings on the skew monoidal category $\RMod[B]$;
\item Braidings on the monoidal category $\RMod^{B}$ of $B$-comodules.
\end{enumerate}
\end{theorem}

\proof
The monoidal category $\RMod$ satisfies the assumptions of
Remark~\ref{rmk:Igenerator}, and so we may use
Theorem~\ref{thm:C^G-C[G]} to deduce the bijection between (2) and (3).
The bijection betwen (1) and (3) is well-known, and can be found for
example in \cite[Proposition~15.2]{Street-book}, on taking $\cv=\RMod\op$.
\endproof

Explicitly, given a braiding $s$ on the skew monoidal category
$\RMod[B]$, the isomorphism $s\colon (I*I)*I\cong (I*I)*I$ amounts
(modulo unit isomorphisms) to an isomorphism $BB\cong BB$, and
composing with $\epsilon\epsilon\colon BB\to I$ gives the corresponding
coquasitriangular structure. Conversely, given coquasitriangular
structure $\sigma\colon BB\to I$, the braiding $(X*Y)*Z\cong (X*Z)*Y$ 
corresponds to the map $XBYBZ\to XBZBY$ given in the string diagram
below. 
\centerline{
\begin{tikzpicture}[scale=1]  
  \path (0,3) node[name=Xu] {$X$}
  (1,3) node[name=B1u] {$B$}
  (2,3) node[name=Yu] {$Y$}
  (3,3) node[name=B2u] {$B$}
  (4,3) node[name=Zu] {$Z$}
  (0,0) node[name=Xd] {$X$}
  (1,0) node[name=B1d] {$B$}
  (2,0) node[name=Zd] {$Z$}
  (3,0) node[name=B2d] {$B$}
  (4,0) node[name=Yd] {$Y$}
  (1,1) node[arr,name=sig] {$\sigma$}
  (1,2.2) node[style=empty, name=d1] {}
  (3,2.2) node[style=empty, name=d2] {};
  \draw[braid] (Xu) to (Xd);
  \path[braid, name path=Y] (Yu) to[out=270, in=90] (Yd); 
  \draw[braid, name path=Z] (Zu) to (4,2.2) to[out=270, in=90] (Zd); 
  \draw[braid] (B1u) to (d1);
  \draw[braid] (B2u) to (d2);
  \draw[braid] (d1) to [out=225, in=135] (sig);
  \path[braid, name path=B1r] (d1) to [out=315, in=90] (B2d); 
  \draw[braid, name path=B2l] (d2) to [out=225, in=45] (sig);
  \draw[braid, name path=B2r] (d2) to [out=315, in=90] (B1d);
  \fill[white,name intersections={of=Y and Z}] (intersection-1) circle(0.1); 
  \fill[white,name intersections={of=Y and B2l}] (intersection-1) circle(0.1); 
  \fill[white,name intersections={of=Y and B2r}] (intersection-1) circle(0.1); 
  \fill[white,name intersections={of=B1r and Z}] (intersection-1) circle(0.1); 
  \fill[white,name intersections={of=B1r and B2l}] (intersection-1) circle(0.1); 
  \fill[white,name intersections={of=B1r and B2r}] (intersection-1) circle(0.1); 
  \draw[braid] (Yu) to[out=270, in=90] (Yd);
  \draw[braid] (d1) to [out=315, in=90] (B2d); 
\end{tikzpicture}
}

\subsection{Duality} 
One can now consider what our results give under the various duality
principles described in Remark~\ref{rmk:duals}. Of particular interest
is $\cc\op$: skew cowarpings on $\cc\op$ are a reverse version of the
skew warpings of \cite{skew}, and any opmonoidal monad $(T,\mu,\eta)$
 on \cc 
gives rise to one. There is a resulting ``warped'' right skew monoidal
category $\cc[T]$ with tensor $X*Y=X.TY$.

It follows formally from
Theorem~\ref{thm:eilenbergMoore} that braidings on the Eilenberg-Moore
category $\cc^T$, with the lifted monoidal structure, are in bijection
with braidings on the opmonoidal monad $T$.

We could also apply Theorem~\ref{thm:C^G-C[G]} to $\cc\op$, but this
is not so interesting, since for the typical choices of $\cc$ the
property of Remark~\ref{rmk:Igenerator} is less likely to
hold for
$\cc\op$. But in fact it is not hard to see that 
Proposition~\ref{prop:special} and Theorem~\ref{thm:C^G-C[G]}
hold for \cc provided that {\em either $\cc$ or $\cc\op$} has
the property of Remark~\ref{rmk:Igenerator}. Thus if \cc has the
property of Remark~\ref{rmk:Igenerator}, then braidings on $\cc^T$
correspond to braidings on the right skew monoidal category $\cc[T]$.

The analogue of Theorem~\ref{thm:bialgebras} then says that for a
commutative ring $R$ and $R$-bialgebra $B$ there are bijections
between:
\begin{enumerate}
\item braidings (quasitriangular structures) on $B$;
\item braidings on the right skew monoidal category $\RMod[B]$;
  \item braidings on the monoidal category $\BMod$ of $B$-modules
\end{enumerate}
and so in particular there is a ``trivial'' braiding, lifted
from the base braided monoidal category, whenever
the bialgebra $B$ is cocommutative.

More generally, consider a bialgebroid $B$ over a (not
necessarily commutative) ring $R$; this amounts to a cocontinuous
opmonoidal monad $T$ on the monoidal category $R\textbf{-Mod-}R$ of
$R$-bimodules. The opmonoidal structure gives rise to a lifted
monoidal structure on the Eilenberg-Moore category; this is just the
category of $B$-modules. The base monoidal category
$R\textbf{-Mod-}R$ is of course not braided. Nonetheless, we have a
bijection between braidings on the Eilenberg-Moore category and
braidings on $T$, giving rise to a notion of braided bialgebroid;
compare this with the quasitriangular structures of \cite{DuninMudrov}. 

\subsection{Skew semiwarpings}
\label{section:non-unital}
~
For $\cc$ skew monoidal, the tensor product $A \ast B = (IA)B$ does
not, in general, form part of a skew monoidal structure on $\cc$ ---
but is \emph{skew semimonoidal} in the sense of \cite[Section~7]{skewcoherence}.  A braiding on $\cc$ then yields natural isomorphisms $A \ast B \to B \ast A$ satisfying the classical braid axioms.  We establish these results in the present section, and will employ them in Section~\ref{sect:bicat} to construct symmetric monoidal bicategories.

We start by considering a generalization of the skew warpings of
\cite{skew} in which all structure involving the unit is omitted. A
{\em skew semiwarping} on a skew monoidal category \cc will be a
functor $T\colon \cc\to\cc$ equipped with a natural transformation
$v\colon T(TA.B)\to TA.TB$ making the diagram
\begin{equation}
  \label{eq:warping}
  \xymatrix{
    T(T(TA.B).C) \ar[r]^-{T(v.1)} \ar[d]_{v} & T((TA.TB).C)
    \ar[r]^-{Ta} & T(TA.(TB.C)) \ar[r]^-{v} & TA.T(TB.C) \ar[d]^{1.v}
    \\
    T(TA.B).TC \ar[r]_-{v.1} & (TA.TB).TC \ar[rr]_-{a} && TA.(TB.TC)
  }
\end{equation}
commute. Just as in \cite[Section~3]{skew}, we may define a ``warped''
tensor product $A\ast B=TA.B$ and the composite
\[ \xymatrix{
    T(TA.B).C \ar[r]^{v.1} & (TA.TB).C \ar[r]^{a} & TA.(TB.C) } \]
defines a morphism $
\alpha\colon (A\ast B)\ast C\to A\ast (B\ast C)$ 
satisfying the pentagon equation. Thus $\cc$ becomes a {\em skew
  semimonoidal category} in the sense of \cite[Section~7]{skewcoherence} with
respect to this warped structure.

We define an {\em augmentation} on
such a skew semiwarping to be a natural transformation
$\epsilon\colon T\to 1$ making the following diagram commute
\[ \xymatrix{
    T(TA.B) \ar[r]^-v \ar[dr]_{\epsilon} & TA.TB \ar[d]^{1.\epsilon}
    \\
    & TA.B } \]
For such an augmentation $\epsilon$, the diagram
\[ \xymatrix{
    (A\ast B)\ast C \ar@{=}[r] \ar[dd]_{\alpha} & T(TA.B).C \ar[r]^-{T(\epsilon.1).1} \ar[d]^{v.1}
    \ar[dr]^{\epsilon.1} & T(A.B).C
    \ar[r]^{\epsilon.1} & (A.B).C \ar[dd]^{a} \\
    & (TA.TB).C \ar[d]^{a} \ar[r]^-{(1.\epsilon).1} & (TA.B).C \ar[d]^a
    \ar[ur]_{(\epsilon.1).1}  \\
  A\ast(B\ast C) \ar@{=}[r] & TA.(TB.C) \ar[r]_-{1.(\epsilon.1)} & TA.(B.C) \ar[r]_-{\epsilon.1} &
A.(B.C) } \]
commutes, and so $\epsilon.1\colon TA.B\to A.B$ is compatible with the
associativity maps $\alpha$ and $a$; in other words, it makes the
identity functor $1\colon \cc\to\cc$ into a semimonoidal functor from
underlying skew semimonoidal category of the original skew monoidal
\cc, to the warped skew semimonoidal category. 

\begin{proposition}
  For a skew monoidal \cc with $T\colon \cc\to \cc$ defined by
  tensoring on the left with $I$, the maps
  \[ \xymatrix{
      I(IA.B) \ar[r]^-{\ell} & IA.B \ar[r]^-{r.1} & (IA)I.B
      \ar[r]^-{a} & IA.IB } \]
  define a skew semiwarping $v$, for  which $\ell\colon IA\to A$ is an augmentation.
  \end{proposition}

\proof
First we verify that $v$ is a skew semiwarping. Observe that $v$ is the instance $X=IA$ of the natural
transformation $w\colon I(X.B)\to X.IB$ given by 
\[ \xymatrix{
I(X.B) \ar[r]^-{\ell} & X.B \ar[r]^-{r.1} & XI.B \ar[r]^-{a} & X.IB
} \]
and therefore that the left hand square in 
\[ \xymatrix{
    I(I(IA.B).C) \ar[r]^-{I(v.1)} \ar[d]_v & I((IA.IB).C) \ar[d]^{v}
    \ar[r]^-{1a} & I(IA.(IB.C)) \ar[r]^-{v} & IA.I(IB.C) \ar[d]^{1.v} 
    \\
    I(IA.B).IC \ar[r]_-{v.1} & (IA.IB).IC \ar[rr]_-{a} && IA.(IB.IC)  } \]
commutes by naturality of $u$; thus it remains to show the
commutativity of the region on the right. This in turn follows by
commutativity of
\[ \xymatrix{
   I(IA.(IB.C)) \ar@/^1pc/[rrr]^{v}
    \ar[r]_-{\ell} & IA.(IB.C) \ar[r]^-{r1} \ar[drr]^1 & (IA)I.(IB.C) \ar[r]_-{a}
    & IA.I(IB.C) \ar[d]_{1.\ell} \ar@<2pc>@/^1pc/[ddd]^{1.v} \\ 
     I((IA.IB). C) \ar[u]^-{1a} \ar[r]^-{\ell} \ar[drrr]_v &  (IA.IB).C
     \ar[rr]^-{a}  && IA.(IB.C) \ar[d]_{1.(r.1)} \\
  &&& IA.((IB)I.C) \ar[d]_{1.a} \\
  &&& IA.(IB.IC). } \]
As for the augmentation property, this holds by commutativity of
\[ \xymatrix{
    I(IA.B) \ar[r]_-{\ell} \ar@/^1pc/[rrr]^v & IA.B \ar[r]_-{r.1}
    \ar@/_1pc/[drr]_1 & (IA)I.B \ar[r]_-a &
    IA.IB \ar[d]^{1.\ell} \\
    &&& IA.B. } \]
\endproof

If now $\cc$ is braided, the natural isomorphisms $s\colon (IA)B\to (IB)A$ can
be seen as natural isomorphisms $c\colon A\ast B\to B\ast A$. We close
this section with a result indicating the sense in which these
$c$ deserve to be thought of as a braiding.

\begin{proposition}\label{prop:2-variable}
  The natural isomorphism $c\colon A\ast B\to B\ast A$ and its inverse
  both satisfy the axiom for a braiding asserting the commutativity of
  \[ \xymatrix{
  (A\ast B)\ast C \ar[r]^-{c\ast 1} \ar[d]_{\alpha} & (B\ast A)\ast C \ar[r]^-{\alpha} &
  B\ast (A\ast C) \ar[d]^{1\ast c} \\
  A\ast (B\ast C) \ar[r]_{c} & (B\ast C)\ast A \ar[r]_-{\alpha} & B\ast (C\ast A). } \]
  If $c$ is a symmetry then the equation
  \[ \xymatrix{
  A \ast B \ar@/_1pc/[rr]_{1}  \ar[r]^{c_{A,B}} & B \ast A \ar[r]^{c_{A,B}} & A \ast B}
\]
also holds.
\end{proposition}

\proof
The second sentence is immediate. The first holds by commutativity of the following diagram
\[ \xymatrix{
    I(IX.Y).Z \ar[r]^{\ell.1} \ar[ddd]_{1s.1} & (IX.Y).Z
    \ar[r]^{(r1)1} \ar[dd]^1
    \ar[dr]^{(r1.1)1}  & ((IX)I.Y).Z
    \ar[r]^{a1} \ar[d]^{(s1)1} & (IX.IY).Z \ar[r]^{a} \ar[dd]^{s1} & IX.(IY.Z) \ar[ddd]^{s} \\
    && ((II)X.Y).Z \ar[d]^{s1} \ar[dl]_{(\ell1.1)1}& &\\
    & (IX.Y)Z \ar[d]^{s1}  & ((II)Y.X).Z \ar[r]^-{a1.1}  \ar[dl]_{(\ell1.1)1}&  (I(IY).X).Z \ar[d]^s \ar[dll]^{\ell1.1}
    & \\
     I(IY.X).Z  \ar[r]^{\ell.1}   & (IY.X)Z  \ar[d]^{(r1)1}  \ar[drr]^s
    &    & (I(IY).Z).X \ar[r]^{a1} \ar[d]^{\ell1.1} &  I(IY.Z).X
    \ar[dl]^{\ell1}\\
    & ((IY)I.X)Z \ar[r]^{s} \ar[d]^{a1} & ((IY)I.Z).X \ar[d]^{a1}& (IY.Z).X \ar[l]^{(r1)1} \\
 & (IY.IX)Z \ar[d]^a & (IY.IZ).X \ar[d]^{a} \\
  & IY.(IX)Z \ar[r]_{1s} & IY.(IZ.X) } \]
and the fact that $s^{-1}$ is also a braiding on the skew monoidal
category \cc. \endproof

\section{Braided skew multicategories}
\label{sect:multicat}

In the earlier paper \cite{bourke-lack}, we defined a notion of {\em
  skew multicategory}, and showed that skew monoidal categories could
be characterized as skew multicategories satisfying a condition we
called {\em left representability}. This is a skew analogue of the
relationship between monoidal categories and representable
multicategories. 

In this section we build on the ideas of \cite{bourke-lack}, 
defining a notion of braiding on a skew multicategory, and showing
that, in the left representable case, this is equivalent to a braiding
on the corresponding skew monoidal category.

We begin by revisiting the notion of skew multicategory defined in
\cite{bourke-lack}, to which we refer for further detail.

\subsection{Skew multicategories}

A {\em skew multicategory} $\bba$ consists  of a
multicategory $\bba_\ell$ together with extra structure which we shall
shortly describe. We generally write $A$ for the set of objects, and
$\overline{a}$ to denote a list $a_1,\ldots,a_n$ of objects, so that
we may write $\bba_\ell(\overline{a};b)$ for the multihom. The
elements of such a multihom are called ``loose multimaps''. We now
turn to the extra structure. This consists of:
\begin{itemize}
  \item
for each {\em non-empty} such list $\overline{a}$ and each $b\in A$, there is a set
$\bba_t(\overline{a};b)$ ``of tight multimaps'', with a function
$j\colon \bba_t(\overline{a};b)\to \bba_\ell(\overline{a};b)$
\item
  for each $a\in A$ there is an element $1_a\in \bba_t(a;a)$ which is
  sent by $j$ to the corresponding identity of $\bba_\ell$
\item
  substitution maps
  \[ \bba_t(b_1,\ldots,b_n;c) \x \bba_t(\overline{a}_1;b_1) \x \prod\limits^n_{i=2} 
\bba_{\ell}(\overline{a}_i;b_i) \to
\bba_t(\overline{a}_1,\ldots,\overline{a}_n;c) \]
whose effect we denote by $(g,f_1,\ldots,f_n)\mapsto g(f_1,\ldots,f_n)$
\end{itemize}
subject to the evident associativity and identity axioms, along with
the requirement that $j$ respect substitution. There is an induced
category \ca of tight unary maps with the same objects as $\bba$ and
with $\ca(a,b)=\bba_t(a;b)$.

For a discussion of various possible reformulations of the notion of
skew multicategory,
see \cite [Sections~3--4]{bourke-lack}.
 
\begin{remark}\label{rk:sub}
In practice, many examples of skew multicategories have the property
that the functions $j\colon \bba_t(\overline{a};b)\to \bba_\ell(\overline{a};b)$ are subset inclusions.  Such
skew multicategories amount to \emph{ordinary multicategories}
equipped with a \emph{subcollection of tight multimaps} which are
non-nullary, contain the identities, and with 
$g(f_{1},\ldots,f_{n})$ tight just when both $g$ and $f_{1}$
are. 
\end{remark}

For the moment we content ourselves with a single example; more will
be given in Section~\ref{sect:2cat} below.

\begin{example}\label{ex:FP}
There is a multicategory $\mathbb {FP}$ of categories equipped with a choice of finite products and whose multimaps
are functors $F\colon \ca_{1} \times \ldots \times \ca_{n} \to \cb$
preserving products in each variable in the usual up to isomorphism
sense.  A nullary map, an element of $\mathbb {FP}( ~;\ca)$, is an object of $\ca$.  Declaring a multimap to be tight just when it preserves products \emph{strictly in the first variable} equips $\mathbb{FP}$ with the structure of a skew multicategory.  
\end{example}

\subsection{Skew monoidal categories arising from left representable
  skew multicategories}

For a skew multicategory $\bba$ and a list $\overline{a}$ of objects,
there are induced functors $\bba_t(\overline{a};-)$ and
$\bba_\ell(\overline{a};-)$ from $\ca$ to $\Set$ (with the former
defined only if $\overline{a}$ is non-empty).
We say that $\bba$ is {\em weakly representable} if these functors
$\bba_\ell(\overline{a};-)$ and $\bba_t(\overline{a};-)$ are representable.
This says that, for $x\in\{t,\ell\}$, there are objects
$m_x\overline{a}$ and multimaps $\theta_{x}(\overline{a}) \in \mathbb
A_{x}(\overline{a};m_{x}\overline{a})$ with the universal property
that the induced functions
\[- \circ_{1} \theta_{x}(\overline{a})\colon\mathbb
  A_{t}(m_{x}\overline{a};b) \to  \bba_{x}(\overline{a};b)\]
  are bijections for all $b \in A$.

If now $\overline{b}$ is a list and $c$ an object, the multimap
$\theta_x(\overline{a})$ induces a function 
\begin{equation}\label{eq:left}
- \circ_{1} \theta_{x}(\overline{a})\colon \mathbb
A_{\myt}(m_{x}\overline{a},\overline{b};c) \to
\bba_{x}(\overline{a},\overline{b};c)
\end{equation}
and the weakly representable $\mathbb A$ is said to be \emph{left representable} if the function
\eqref{eq:left} is a bijection for all $x,\overline{a},\overline{b}$ and $c$,
and all universal multimaps $\theta_x(\overline{a})$.

Theorem 6.1 of \cite{bourke-lack} asserts that there is a 2-equivalence between the 2-categories of left representable skew multicategories and of skew monoidal categories.  We now describe the skew monoidal structure on $\ca$ associated to the left representable $\bba$.

Setting $AB=m_{t}(A,B)$ gives the defining representation 
\begin{equation}\label{eq:tensor}
\ca(AB,C) \cong \bba_{t}(A,B;C)
\end{equation}
with universal multimap denoted $e_{A,B} \in  \bba_{t}(A,B;AB)$.  We sometimes write it as
$$e\colon A,B \to AB$$
omitting the subscript.

Given $f\colon A \to C$ and $g\colon B \to D$ in $\ca$, the
morphism $fg\colon AB \to CD$ is the unique one such that
\begin{equation}\label{eq:tensoring}
fg \circ e_{A,B} = e_{C,D}(f,g) 
\end{equation}
which condition we also write as 
\begin{equation} 
\xymatrix{
A,B \ar[d]_{e} \ar[rr]^{f,g} && C,D \ar[d]^{e} \\
AB \ar[rr]_{fg} && CD .
}
\end{equation}
Functoriality follows from the universal property.

By left representability we have natural isomorphisms $\ca((AB)C,D)
\cong \bba_{t}(AB,C;D) \cong \bba_{t}(A,B,C;D)$.  The universal
multimap is the composite $e_{AB,C} \circ_1 e_{A,B} \in
\bba_{t}(A,B,C;(AB)C)$, which we may represent as 
\begin{equation}
\xymatrix{
A,B,C \ar[r]^{e,1} & AB,C \ar[r]^{e} & (AB)C.
}
\end{equation}
By its universal property we obtain the associator $a_{A,B,C}\colon
(AB)C \to A(BC)$ as   the unique map such that
\begin{equation}\label{eq:associator}
 a_{A,B,C} \circ_{1} e_{AB,C} \circ_{1} e_{A,B} = e_{A,BC} \circ_1 e_{B,C}
\end{equation}
or equally
\begin{equation}
\xymatrix{
A,B,C \ar[d]_{e} \ar[r]^{e,1} & AB,C \ar[r]^{e} & (AB)C \ar[d]^{a} \\
A,BC \ar[rr]^{e} && A(BC). 
}
\end{equation}

We define  the unit $I$   as the representing object
for $\bba_{l}(;-)\colon \ca \to \Set$. We write
$u \in \bba_{l}(;-)$ for the universal multimap, and depict it as
$u\colon (-) \to I$. By left representability we have $\ca(
IA ,B) \cong \bba_{t}(I,A;B) \cong \bba_{l}(A;B)$.  Taking $B=A$
and the image of the identity $1_{a}$ under $\bba_{t}(A;A) \to
\bba_{l}(A;A)$ yields $\ell\colon IA \to A$, the unique map such that
 \begin{equation*}
 \xymatrix{A \ar@/_1.5pc/[rrr]_{j1_{A}} \ar[r]^-{u,1} & I,A \ar[r]^{e} & IA \ar[r]^{l} & A}
 \end{equation*}
 commutes.  The right unit map $r\colon A \to AI$ is the composite
 below:
   \begin{equation*}
 \xymatrix{A \ar[r]^-{1,u} & A,I \ar[r]^-{e} & AI . }
 \end{equation*}

  \subsection{Braided skew multicategories}
 
 We begin by recalling braided multicategories. These differ from the usual notion of symmetric multicategory in that they involve actions of the braid groups $\cb_{n}$ rather than the symmetric groups $\cs_{n}$. 
 
\subsubsection*{Braid groups and symmetric groups}
 
 Recall that the Artin braid group $\cb_{n}$ has presentation
 $$\langle \beta_{1},\ldots,\beta_{n-1} | \beta_{i}\beta_{j}=\beta_{j}\beta_{i} \textnormal{ for } j < i-1, \beta_{i}\beta_{i+1}\beta_{i+1}=\beta_{i}\beta_{i+1}\beta_{i}\rangle .$$
There is an evident homomorphism $|-|_{n}\colon \cb_n \to \cs_n$ sending $\beta_i$ to the transpostion $(i,i+1)$ so that, in particular, $\cb_n$ acts on $\{1,\ldots,n\}$.

In addition to the group operation, one can form the tensor product of braids. 
Combining this with the group operations, the sets $\cb_n$
admit an evident substitution $$\cb_n \times \cb_{m_{1}} \times \ldots
\cb_{m_{n}} \to \cb_{m_{1}+\ldots + m_{n}}\colon
(s,(t_{1},\ldots,t_{n})) \mapsto s(t_{1},\ldots,t_{n})$$ which,
indeed, form the substitution for an operad $\cb$, and the
functions $|-|_n\colon \cb_n\to\cs_n$ define an operad morphism from
$\cb$ to the corresponding operad \cs.

\subsubsection*{Braided multicategories}
A \emph{braiding} on a multicategory $\bba$ consists of 
\begin{enumerate}
\item for each $s \in \cb_{n}$ a function
  \[ s^*\colon \bba(a_{1},\ldots,a_{n};b) \to
    \bba(a_{s1},\ldots,a_{sn};b) \colon f\mapsto fs \]
   satisfying  the action equations  $(fs)t = f(st)$
  and $f1_{\cb_n}=f$ as well as 
\item the equivariance equation
 \begin{equation*}\label{eq:equivariance}
(f(g_{1},\ldots,g_{n}))s(t_{1},\ldots,t_{n}) = fs(g_{s_{1}}t_{s_{1}},\ldots,g_{s_{n}}t_{s_{n}})
\end{equation*}
 for all $f \in \bba(b_1,\ldots,b_{n};c)$ and $s \in \cb_n$, together with $g_{i} \in \bba(\overline{a_{i}};b_i)$ and $t_{i} \in \cb_{|\overline{a_{i}}|}$ for $i \in \{1,\ldots,n\}$.
\end{enumerate}

The braiding is a symmetry if the actions  satisfy $s^*= t^*$ whenever
$|s|=|t|$ 
Alternatively, and more simply, modify the definition above by replacing each occurence of  $\cb$ by $\cs$.

 For concrete calculations it may be useful to reformulate this
structure in terms of the generating braids. To give bijections
satisfying (1) is to give bijections $f\mapsto f\beta_i$ for each
$i$, subject to the braid relations
$(f\beta_i)\beta_j=(f\beta_j)\beta_i$ for $j<i-1$ and
$((f\beta_i)\beta_{i+1})\beta_i =
((f\beta_{i+1})\beta_i)\beta_{i+1}$.
The equivariance conditions become more complicated. Given $f\in
\bba(a_1,\ldots,a_n;b_i)$ and $g\in \bba(b_1,\ldots,b_m;c)$ they say
\begin{equation}\label{eq:58}
  g\circ_i f\beta_j = (g\circ_i f)\beta_{i+j-1}
  \end{equation}
  and
  \begin{equation}\label{eq:b-e}
 g\beta_j\circ_i f =
  \begin{cases}
    (g\circ_i f)\beta_j & \text{if $j<i-1$} \\
    (g\circ_{i-1}f)\beta_{j+n-1}\ldots \beta_{j} & \text{if $j=i-1$} \\
    (g\circ_{i+1}f)\beta_j\ldots\beta_{j+n-1} & \text{if $j=i$} \\
    (g\circ_i f)\beta_{j+n-1} & \text{if $j>i$.}
  \end{cases}
\end{equation}

\subsubsection*{Braided skew multicategories}

Now let $\bba$ be a skew multicategory.  For a braiding on $\bba$ we
require, to begin with, that the ordinary multicategory $\bba_l$ of
loose multimaps be equipped with actions $$s^*\colon\bba_{l}(a_{1},\ldots,a_{n};b) \to \bba_{l}(a_{s1},\ldots,a_{sn};b)$$ exhibiting it as a braided multicategory.

Consider the subgroup $\cb^1_{n} = \langle \beta_{2},\ldots,\beta_{n} \rangle \leq \cb_{n}$; that is, we omit the single generator having a non-trivial action on $1 \in \{1,\ldots,n\}$.  Of course $\cb^1_n \cong \cb_{n-1}$.  Observe also that $s(t_{1},\ldots,t_{n}) \in\cb^1_{m_{1}+\ldots m_{n}}$ whenever $s \in \cb^1_{n}$ and $t_{1} \in \cb^1_{m_{1}}$. 
In a braided skew multicategory we also require:

\begin{enumerate}
\item[(1*)] for each $s \in \cb^1_{n}$ a function $
  s^*\colon 
  \bba_t(a_{1},\ldots,a_{n};b) \to \bba_t(a_{s1},\ldots,a_{sn};b)$
  such that  these satisfy the action equations $(fs)t = f(st)$, $f1=f$ as well as the compatibility $j(fs)=j(f)s$.
\item[(2*)] given $f \in \bba_{t}(b_1,\ldots,b_{n};c)$, $s \in \cb^1_n$, $g_{1} \in \bba_t(\overline{a_{1}};b_1)$ and $t_1 \in \cb^1_{m_{1}}$, plus $g_{i} \in \bba_l(\overline{a_{i}};b_i)$, $t_{i} \in \cb_{m_{i}}$ for $i \in \{2,\ldots,n\}$, we require the equivariance equation
\begin{equation*}\label{eq:equivariance2}
(f(g_{1},\ldots,g_{n}))s(t_{1},\ldots,t_{n}) = fs(g_{s_{1}}t_{s_{1}},\ldots,g_{s_{n}}t_{s_{n}}).
\end{equation*}
\end{enumerate}
 Once again this can be reformulated in terms of the generators.
 There are assignments $g\mapsto g\beta_i$ for $g$ tight $n$-ary and
 $1<i<n$; and $g\mapsto g\beta_i$ for $g$ loose $n$-ary and $1\le
 i<n$. These satisfy the braid relations as well as equations
 \eqref{eq:58} and \eqref{eq:b-e} insofar as these make sense. Finally
 the two actions should be compatible in the sense that
 $j(g\beta_i)=j(g)\beta_i$ for $g$ tight $n$-ary and $1<i<n$.
 
For a symmetric skew multicategory we also require that $s^*=t^*$ whenever $|s|=|t|$ and $s,t \in \cb^1_n$.  

Alternatively, letting $\cs^{1}_{n} \subseteq \cs_{n}$ denote the subgroup of permutations fixing $1 \in \{1,\ldots,n\}$,  we obtain a simpler definition of symmetric skew multicategory, by replacing each appearance of  $\cb$ by $\cs$.

\begin{remark}\label{rk:sub2}
Recall from Remark~\ref{rk:sub} that a skew multicategory $\bba$ for which the comparison functions
$j_{\overline{a},b}\colon \mathbb A_{t}(a_{1},\ldots,a_{n};b) \to \mathbb A_{l}(a_{1},\ldots,a_{n};b)$
are inclusions amounts to an ordinary multicategory equipped with a
subcollection of tight morphisms which are not nullary, contain the
identities and have the property that $f(g_{1},\ldots,g_{n})$ is tight
just when $f$ and $g_{1}$ are.  Under this correspondence a braiding
on the skew multicategory simply amounts to a braiding on the
associated multicategory with the property that if $s \in \cb^{1}_{n}$
and $f$ is a tight multimap of arity $n$ then $ fs$ is tight too.
 There is a corresponding result for symmetries with $\cb^{1}_{n}$ replaced by $\cs^{1}_{n}$.
\end{remark}

\begin{example}
The multicategory $\mathbb{FP}$ of Example~\ref{ex:FP} admits a
symmetry lifted directly from the cartesian multicategory $\bCat$.
For if $s \in \cs_n$ then $Fs \colon \ca_{s1} \times \ldots \times
\ca_{sn} \to \cb$ will preserve products in the $i$th  variable just
when $F$ preserves products in the $s_i$th variable.
Since the tight multimaps in $\mathbb{FP}$ are defined to be those
preserving products strictly in the first variable $Fs$ will be tight
so long as $F$ is and $s \in \cs^1_n$.
Accordingly $\mathbb{FP}$ is a symmetric skew multicategory.
\end{example}

\subsection{Braidings on left representable skew multicategories}
\label{sect:left-rep-braiding}

Suppose that \bba is a left representable skew multicategory
corresponding to a skew monoidal category \ca.  Given a braiding on
\bba,  
the generator $\beta_2 \in \cb^1_3$ induces a bijection $\beta_2^{*} \colon \bba_{t}(A,C,B;D) \to \bba_{t}(A,B,C;D)$.   By equivariance these isomorphisms are natural in $\ca$, as are the vertical isomorphisms below.
\begin{equation}\label{eq:beta2}
\xymatrix{
\ca((AC)B;D) \ar[d]_{\cong} \ar[rr]^{\ca(s_{A,B,C},D)}&&  \ca((AB)C,D) \ar[d]^{\cong} \\
\bba_{t}(A,C,B;D) \ar[rr]^{\beta_2^{*}} && \bba_{t}(A,B,C;D).
}
\end{equation} 
By the Yoneda Lemma there is a unique natural isomorphism 
\[ s_{A,B,C}\colon (AB)C \to (AC)B \]
rendering the diagram above commutative.

In the appendix, we prove:

\begin{theorem}\label{thm:multicat}
  Let \ca be a skew monoidal category, and \bba the corresponding left
  representable skew multicategory. The above assignment defines a
  bijection between braidings on \bba and braidings on \ca, which
 restricts to a bijection between symmetries on \bba and symmetries
  on \ca. 
\end{theorem}

\section{Symmetric skew monoidal 2-categories and symmetric monoidal bicategories}
\label{sect:2cat}

This last section is geared towards understanding the braidings on
skew monoidal 2-categories, like $\FProds$ as well as (bicategorical) braidings on the induced
monoidal bicategories. We do this using the
corresponding skew multicategories, and the results of the previous section.

We have not, thus far, mentioned 2-categorical (that is,
$\Cat$-enriched) structure.  The various structures that we have been
dealing with --- skew monoidal categories, skew multicategories and
their braided and symmetric variants --- each admit $\Cat$-enriched
analogues.  For instance a skew monoidal 2-category involves a tensor
product 2-functor, as well as 2-natural transformations $a,\ell,
 r$ all satisfying the usual five equations.  For a braided
(or symmetric) skew monoidal 2-category we of course
require that the braiding be 2-natural as well.  In a skew
2-multicategory $\mathbb A$ one has \emph{categories} of tight and
loose multimaps connected by a \emph{functor} $j\colon
\bba_t(\overline{a};b)\to \bba_\ell(\overline{a};b)$ and, again, the
substitution maps themselves become functors rather than just
functions.  In the braided/symmetric variants the actions $s^{*}$ on
the multihoms are themselves functors.  Again we can speak of
$\Cat$-enriched left representability --- which now involves
isomorphisms of categories rather than mere bijections of sets.

Just as in the unenriched setting, skew 2-multicategories in which each $j\colon \bba_t(\overline{a};b)\to \bba_\ell(\overline{a};b)$ is the inclusion of a full subcategory (rather than a mere subset) can be identified with 2-multicategories equipped with a subcollection of non-nullary tight multimorphisms having the same closure properties described in \ref{rk:sub}.  Again, a braiding/symmetry in this context simply amounts to a braiding/symmetry on the 2-multicategory which respects tight multimaps in the sense described in \ref{rk:sub2}.

$\mathbb{FP}$ is a simple example of such a symmetric skew 2-multicategory --- for each of $x=t,l$ the morphisms of $\mathbb{FP}_{x}(\ca_1,\ldots,\ca_n;\cb)$ are the natural transformations.

\begin{examples}\label{ex:monad}
More generally any pseudocommutative 2-monad $T$ on $\Cat$ \cite{Hyland2002Pseudo} gives rise to a skew 2-multicategory $\TMAlg$.  When $T$ is the 2-monad for categories with finite products we obtain the skew 2-multicategory described in Example~\ref{ex:FP}.  But there are many more examples --- the 2-monads for  permutative categories, symmetric monoidal categories and categories equipped with a given class of limits (or colimits) are all pseudocommutative \cite{Hyland2002Pseudo, Lopez-Franco2011Pseudo}.

For such a $T$ an object of the skew 2-multicategory $\TMAlg$ is a strict $T$-algebra $\mathbf A$; we write $\ca$ for the underlying category of the $T$-algebra $\mathbf A$. A multimorphism $F\colon (\mathbf
A_{1},\ldots, \mathbf A_{n}) \to \mathbf B$ is a functor $F \colon   \ca_1
\times \ldots \times \ca_n \to \cb$ equipped with the structure of an algebra
pseudomorphism in each variable separately, with these $n$
pseudomorphism structures commuting with each other in the sense
explained in \cite{Hyland2002Pseudo}.  Nullary morphisms $a\colon (-)
\to \mathbf A$ are just objects $a$ of $\ca$. There are also
{\em transformations} between the multimorphisms: these are natural
transformations which are $T$-algebra transformations in each variable
separately; once again, see \cite{Hyland2002Pseudo} for the details.

We obtain a skew 2-multicategory \TMAlg by defining a multimap
$F$ as above to be tight if it is a strict algebra morphism in the first variable --- that
is, if for all $a_{2} \in \ca_{2},\ldots,a_{n} \in \ca_{n}$ the
pseudomorphism $F(-,a_{2},\ldots,a_{n})\colon \mathbf A_{1} \to
\mathbf B$ is strict.  See Section 4.2 of \cite{bourke-lack} and the references therein for more on this example.
\end{examples}

\begin{proposition}\label{prop:pseudo}
Let $T$ be an accessible symmetric pseudo-commutative 2-monad on $\Cat$.  Then the skew $2$-multicategory $\TMAlg$ of Example~\ref{ex:monad} is symmetric.  It is also left representable and closed.
\end{proposition}
\begin{proof}
Left representability and closedness of the underlying skew multicategories are established in Examples 4.8 of \cite{bourke-lack}. The additional 2-categorical aspects that must be verified are straightforward. 
The underlying 2-multicategory of loose maps is symmetric by Proposition 18 of \cite{Hyland2002Pseudo}.   A multimap $F\colon (\mathbf
A_{1},\ldots, \mathbf A_{n}) \to \mathbf B$ involves a functor $\ca_{1} \times \ldots \ca_{n} \to \cb$ equipped with a pseudomap structure on the functor $$F(a_{1},\ldots,a_{i-1},-,a_{i+1},\ldots,a_{n})\colon \ca_{i} \to \cb $$for each $i \in \{1,\ldots,n\}$ and tuple
$(a_{1},\ldots,a_{i-1},a_{i+1},\ldots,a_{n})$.  The different
pseudomap structures are required to satisfy compatibility axioms.
The symmetry $s \in \cs_n$ permutes the variables and the pseudomap
structures --- in particular, if $s \in \cs^1_n$ the multimap $Fs
\colon (\mathbf A_{1},\mathbf A_{s2}\ldots, \mathbf A_{sn}) \to
\mathbf B$ has pseudomap $(Fs)(-,a_{2},\ldots,a_{n})\colon
\mathbf A_{1} \to \mathbf B$ given by
$F(-,a_{s^{-1}2},\ldots,a_{s^{-1}n})\colon \mathbf A_{1} \to \mathbf
B$ which is strict whenever $F$ is strict in the first variable.
Hence $Fs$ is tight if $F$ is tight and $s\in \cs^1_n$,  as required.
\end{proof}

\begin{corollary}\label{cor:monad}
Let $T$ be an accessible symmetric pseudo-commutative 2-monad on
$\Cat$.  Then the 2-category $\TAlg_s$ of strict algebras and strict morphisms admits a closed symmetric skew monoidal structure.
\end{corollary}
\begin{proof}
The construction of a symmetric skew monoidal category from a left representable symmetric skew multicategory given in Theorem~\ref{thm:multicat} admits an evident $\Cat$-enriched version.  Applying this, together with Proposition~\ref{prop:pseudo}, we obtain the desired symmetric skew monoidal structure on the 2-category  $\TAlg_s$. For closedness we combine
Proposition~\ref{prop:pseudo} above and Theorem 6.4 of
\cite{bourke-lack} (in its $\Cat$-enriched form).
\end{proof}
All of the pseudo-commutative 2-monads described in
Examples~\ref{ex:monad} are accessible symmetric pseudo-commutative
\cite{Hyland2002Pseudo, Lopez-Franco2011Pseudo}.  Accordingly it
follows from Corollary~\ref{cor:monad} that the categories of
permutative categories, of symmetric monoidal categories, and of
categories with a given class of limits or a given class of
colimits all admit closed symmetric skew monoidal structures.

As a special case of this we obtain the symmetric skew structure on
$\FProds$ described in the introduction, in which maps $\mathcal A
\mathcal B \to \mathcal C$ correspond to functors $\mathcal A
\times\mathcal B \to \mathcal C$ preserving products strictly in the
first variable and up to isomorphism in the second.
The internal hom $[\mathcal A,\mathcal B]$ is the usual category of finite product preserving functors (in the up to isomorphism sense) with products pointwise as in $\mathcal B$.

\subsection{Braided monoidal bicategories}\label{sect:bicat}

Finally, we explain how our 2-categorical examples give rise to
symmetric monoidal bicategories \cite{BaezNeuchl,  Gurski-BraidedMonoidalBicategories}.

To begin with we observe that if $\cc$ is a skew monoidal 2-category
whose coherence constraints are equivalences, then $\cc$ is, in
particular, a monoidal bicategory.  (This is the approach of Section
6.4.3 of \cite{bourko-skew}, on which we now build.)  To equip our
monoidal bicategory with a symmetry, we need to provide pseudonatural
equivalences $AB \to BA$ together with, replacing the
symmetry equations, certain invertible modifications satisfying a host
of coherence axioms.  Now the 2-variable morphisms $AB \to BA$ are not part of the structure that we are, in the skew setting, presented with.  However since the components $\ell_{A}B:A \ast B = (IA)B \to AB$ are equivalences, we can obtain a symmetry of the required form as the composite 
\begin{equation}\label{eq:bicatbraiding}
\xymatrix{
AB \ar[r]^-{{\ell^*_{A}}B} & (IA)B \ar[r]^{s_{I,A,B}} & (IB)A \ar[r]^{\ell_{B}A} & BA
}
\end{equation}
in which ${\ell^*_{A}}:A \to IA$ is the equivalence inverse of $\ell_{A}$.  The composite equivalence will be the desired component of the braiding on our monoidal bicategory.
\begin{proposition}\label{prop:bmb}
Let $\cc$ be a symmetric skew monoidal 2-category whose coherence constraints are equivalences.  Then $\cc$ is a symmetric monoidal bicategory with braiding components as in \eqref{eq:bicatbraiding}.
\end{proposition}

\begin{proof}
Let us firstly note that the symmetry axioms for a symmetric monoidal bicategory do \emph{not} refer to the unit.  Using the prefix \emph{semi}, as usual, to specify that part of a structure not mentioning the unit, it follows that a symmetry on a monoidal bicategory amounts to a symmetry on its underlying semi-monoidal bicategory.  

Now the analysis of Section~\ref{section:non-unital} applies equally to this \Cat-enriched context, and so we obtain skew semimonoidal 2-category with product $A\ast B=IA.B$ and a semimonoidal 2-functor $(1,\ell1):(\cc,.)\to(\cc,\ast)$.    Since the associators for $.$ are equivalences and since the equivalences $A \ast B \to AB$ commute with the respective associators it follows, by 2 out of 3 for equivalences, that the associators for $(\cc,\ast)$ are equivalences --- thus $(\cc,\ast)$ is a semi-monoidal bicategory too.  Furthermore, since by Proposition~\ref{prop:2-variable} the classical symmetry axioms holds on the nose, $(\cc,\ast)$ is in fact a symmetric semi-monoidal bicategory.  

Finally we transport the braiding across the componentwise equivalence $(1,l1):(\cc,.) \to (\cc,\ast)$ of semi-monoidal bicategories to obtain the structure of a symmetric semi-monoidal bicategory on $(\cc,.)$.  The resulting symmetry, obtained by transport of structure, is that described in \eqref{eq:bicatbraiding}.
\end{proof}

Now adapting the approach of Section 6.4.3 of \cite{bourko-skew} let
$(\TAlg_{s})_{c}$ denote the full sub 2-category of $\TAlg_{s}$
containing the flexible (cofibrant) $T$-algebras.  By \cite
[Proposition 6.5]{bourko-skew} if $T$ is an accessible
pseudo-commutative 2-monad on $\Cat$ then the skew monoidal structure
on $\TAlg_s$ restricts to $(\TAlg_{s})_{c}$ where the coherence
constraints become equivalences.  If moreover $T$ is symmetric
pseudo-commutative then, by Corollary~\ref{cor:monad},
$(\TAlg_{s})_{c}$ is a symmetric skew monoidal 2-category whose
coherence constraints are equivalences --- and so, by
Proposition~\ref{prop:bmb} above, it is a symmetric monoidal bicategory.  Lastly, as in Lemma 6.6 of \cite{bourko-skew}, the composite inclusion $(\TAlg_{s})_{c} \to \TAlg_{s} \to \TAlg$ of the flexible algebras and strict morphisms into the 2-category of strict algebras and pseudomorphisms is a biequivalence of 2-categories; transporting the structure of a symmetric monoidal bicategory along the biequivalence we obtain the symmetric monoidal bicategory structure on $\TAlg$.  That is:

\begin{theorem}
Let $T$ be an accessible symmetric pseudo-commutative 2-monad on $\Cat$.  Then the 2-category $\TAlg$ admits the structure of a symmetric monoidal bicategory.
\end{theorem}
 
\appendix

\section{Proof of Theorem~\ref{thm:multicat}}

\label{sect:left-rep-braiding}

In this appendix we give the remaining details in the comparison
between braidings on skew monoidal categories and braidings on the
corresponding left representable skew multicategory.

Suppose that \bba is a left representable skew multicategory
corresponding to a skew monoidal category \ca. Commutativity of the diagram
\begin{equation*}
\xymatrix{
\ca((ac)b;d) \ar[d]_{\cong} \ar[rr]^{\ca(s_{a,b,c},d)}&&  \ca((ab)c,d) \ar[d]^{\cong} \\
\bba_{t}(a,c,b;d) \ar[rr]^{\beta_2^{*}} && \bba_{t}(a,b,c;d)
}
\end{equation*} 
determines a bijection between isomorphisms $s_{a,b,c}:(ab)c \to (ac)b$ and invertible actions 
$g\mapsto
g\beta_2$ for each tight ternary $g$,
 subject to the requirement that  $f\circ_1g\beta_2=(f\circ_1
 g)\beta_2$ for all tight unary $f$. 
 The first step is to understand braidings on \ca in terms of tight
 maps in $\bba$.
 
\begin{theorem}\label{thm:tight-parts}
There is a bijection between braidings on the skew monoidal category
\ca and invertible actions $g\mapsto g\beta_j$ for tight $m$-ary $g$
and $1<j<m$, subject to the braid equations and satisfying the
conditions
\begin{align}
  \label{eq:a}
  f\circ_i g\beta_j &= (f\circ_i g)\beta_{i+j-1}  \\
  \label{eq:b}
  g\beta_j\circ_i f &= (g\circ_i f)\beta_j \qquad\qquad \text{if $j<i-1$} \\
  \label{eq:c}
  g\beta_j\circ_i f &= (g\circ_{i-1}f)\beta_{j+n-1}\ldots \beta_{j} \quad \text{if $j=i-1$} \\
  \label{eq:d}
  g\beta_j\circ_i f &= (g\circ_{i+1}f)\beta_j\ldots\beta_{j+n-1} \quad \text{if $j=i$} \\
  \label{eq:e}
  g\beta_j\circ_i f &=     (g\circ_i f)\beta_{j+n-1} \qquad\qquad \text{if $j>i$}
\end{align}
for tight $n$-ary $f$. Under this bijection the braiding $s$
 corresponds as above to the action $g\mapsto g\beta_2$ for tight
 ternary $g$.
\end{theorem}

Suppose given isomorphisms $s_{a,b,c}\colon (ab)c\to (ac)b$ and
the corresponding actions $g\mapsto g\beta_2$ for tight ternary
$g$. Then \eqref{eq:a} holds for $n=1$ and $m=3$ (in which case
necessarily $i=1$ and $j=2$).
Naturality of $s$ is equivalent to \eqref{eq:c}, \eqref{eq:d}, and
\eqref{eq:e} for $n=1$ and $m=3$ (in which case necessarily
$j=2$). Note also that when $n=1$ and $m=3$ condition \eqref{eq:b} is empty.

Given objects
$a,b,c\in\bba$, let $\theta^{abc}_3\in\bba_t(a,b,c;(ab)c)$ be the
universal tight ternary map.  Postcomposing this by $s_{a,b,c}\colon
(ab)c\to (ac)b$ 
equally gives the map $\theta^{acb}_3\beta_2$. An arbitrary tight ternary
map  $g\in\bba_t(a,b,c;x)$ has the form $g'\circ_1\theta^{abc}_3$ for a
unique tight unary $g'$, and we then have
$g\beta_2=g'\circ_1\theta^{abc}_3\beta_2$.  In the following we shall
usually omit the superscripts in the maps $\theta^{abc}_3$. 

Now for $n \geq 3$ each tight
$n$-ary $g$ has the form  $g=g'\circ_1\theta_3$ for a unique tight
 $n-2$ -ary $g'$. We may now define $g\beta_2$ to be
$g'\circ_1\theta_3\beta_2$, and deduce that
$h\circ_1\theta_3\beta_2=(h\circ_1\theta_3)\beta_2$ for all tight $h$.

\begin{proposition} 
The special case
 \begin{equation}
  \label{eq:circ1beta2}
  f\circ_1 g\beta_2 = (f\circ_1 g)\beta_2 \hspace{0.1cm} 
\end{equation}
of \eqref{eq:a} holds for all tight $f$
and $g$. 
Furthermore the action is uniquely determined by that property
together with the values of $\theta_3\beta_2$.
\end{proposition}

\proof
 Uniqueness follows on taking $g=\theta_3$. To verify
\eqref{eq:circ1beta2} in general, write
$g=g'\circ_1\theta_3$; then
 
  \[ f\circ_1g\beta_2=f\circ_1
  (g'\circ_1\theta_3\beta_2)=(f\circ_1g')\circ_1\theta_3\beta_2 = 
  ((f\circ_1 g')\circ_1\theta_3)\beta_2 =(f\circ_1g)\beta_2.\]
\endproof

\begin{proposition}\label{prop:extended}
  Suppose that the equations \eqref{eq:c}--~\eqref{eq:e} hold for
  all tight ternary $g$ and tight unary $f$. Then they hold for all
  tight $g$ and tight unary $f$ when $j=2$.
\end{proposition}

\proof This follows by
\begin{align*}
  g\beta_2\circ_1 f &= (g'\circ_1\theta_3\beta_2)\circ_1 f \\
                    &= g'\circ_1(\theta_3\beta_2\circ_1f) \\
                    &= g'\circ_1 (\theta_3\circ_1f)\beta_2 \tag{by
                      \eqref{eq:e}}\\
  &= (g'\circ_1(\theta_3\circ_1f))\beta_2 \tag{by
    \eqref{eq:circ1beta2}} \\
                    &= (g\circ_1 f)\beta_2 \\
  g\beta_2\circ_2 f &= (g'\circ_1\theta_3\beta_2)\circ_2 f \\
                    &= g'\circ_1(\theta_3\beta_2\circ_2 f) \\
                    &= g'\circ_1(\theta_3\circ_3 f)\beta_2 \tag{by
                      \eqref{eq:d}}\\
  &= (g'\circ_1(\theta_3\circ_3 f))\beta_2 \tag{by
    \eqref{eq:circ1beta2}} \\
                    &= (g\circ_3 f)\beta_2 \\
  g\beta_2\circ_3 f &= (g'\circ_1\theta_3\beta_2)\circ_3 f \\
                    &= g'\circ_1 (\theta_3\beta_2\circ_3 f) \\
                    &= g'\circ_1(\theta_3\circ_2 f)\beta_2 \tag{by
                      \eqref{eq:c}}\\
  &= (g'\circ_1(\theta_3\circ_2f)\beta_2 \tag{by
    \eqref{eq:circ1beta2}} \\
  &= (g\circ_2f)\beta_2.  \qedhere
\end{align*}

In what follows we let $\theta_m\in\bba_t(a_1,\ldots,a_m;a_1\ldots
a_m)$ be a universal tight $m$-ary map with $m>3$.  We can now define
further actions $g\mapsto g\beta_j$ as follows.

\begin{proposition}\label{prop:eq-e}
  There is a unique assignment $g\mapsto g\beta_j$ for a tight $m$-ary
  $g$ and $2\le j<m$, such that \eqref{eq:e} holds for all tight
  $f$. 
\end{proposition}

\proof
For the uniqueness, take $f=\theta_{j-1}$, to see that
\begin{equation}
  \label{eq:defining}
 (g''\circ_1\theta_{j-1})\beta_j = g''\beta_2\circ_1\theta_{j-1}
\end{equation}
and use the fact that any tight $g$ can be written as
$g''\circ_1\theta_{j-1}$ for a unique tight $g''$.

If $f$ is an arbitrary tight $n$-ary map, write
$\theta_{j-1}\circ_i f=f'\circ_1\theta_{n+j-2}$, where $f'$ is tight
unary, and now
\begin{align*}
  g\beta_j\circ_i f &= (g''\beta_2\circ_1\theta_{j-1})\circ_i f\\
                    &= g''\beta_2\circ_1 (\theta_{j-1}\circ_i f) \\
                    &= g''\beta_2\circ_1 (f'\circ_1\theta_{n+j-2}) \\
                    &= (g''\beta_2\circ_1 f')\circ_1\theta_{n+j-2} \\
                    &= (g''\circ_1 f')\beta_2\circ_1\theta_{n+j-2}
                      \tag{Proposition~\ref{prop:extended}}\\
                    &= ((g''\circ_1
                      f')\circ_1\theta_{n+j-2})\beta_{n+j-1} \tag{by defn}\\
                    &= (g''\circ_1(f'\circ_1\theta_{n+j-2})\beta_{n+j-1} \\
                    &= (g''\circ_1(\theta_{j-1}\circ_i f))\beta_{n+j-1} \\
  &= (g\circ_1 f)\beta_{n+j-1} 
\end{align*}
gives the result.  
\endproof
 
\begin{proposition}\label{prop:s-natural}
  There is a bijection between natural isomorphisms $s$ and invertible
  assignments $g\mapsto g\beta_j$ for tight $m$-ary $g$ and for $2\le
  j<m$, subject to equation \eqref{eq:a} for $i=1$ and  all tight
  $g$ and $f$; \eqref{eq:b} and \eqref{eq:e} for all tight $g$
  and $f$;  \eqref{eq:c} and \eqref{eq:d} for
  all tight $g$ and tight unary $f$.
\end{proposition}

\proof
We have proved \eqref{eq:e} in Proposition~\ref{prop:eq-e}.
The relevant parts of \eqref{eq:a} hold by the calculation
\begin{align*}
  f\circ_1g\beta_j &= f\circ_1(g''\beta_2\circ_1\theta_{j-1}) \\
                   &= (f\circ_1g''\beta_2)\circ_1\theta_{j-1} \\
  &= (f\circ_1g'')\beta_2\circ_1\theta_{j-1} \tag{by \eqref{eq:circ1beta2}} \\
                   &= ((f\circ_1 g'')\circ_1\theta_{j-1})\beta_j \tag{by \eqref{eq:defining}} \\
  &= (f\circ_1 g)\beta_j.
\end{align*}

For \eqref{eq:d} with tight unary $f$ we have
 \begin{align*}
g\beta_j\circ_j f&= (g''\beta_2\circ_1\theta_{j-1})\circ_j f \\
                        &= (g''\beta_2 \circ_2 f)\circ_1\theta_{j-1} \\
                        &= (g''\circ_3 f)\beta_2 \circ_1 \theta_{j-1}
                          \tag{Proposition~\ref{prop:extended}} \\
                        &= ((g''\circ_3 f) \circ_1
                          \theta_{j-1})\beta_j \tag{by \eqref{eq:defining}}\\
                        &= (g''\circ_1 \theta_{j-1}) \circ_{j+1} f)\beta_j\\
                        &= (g \circ_{j+1} f)\beta_{j}.
 \end{align*}
  A similar argument gives \eqref{eq:c} for tight unary
 $f$.  
Finally for \eqref{eq:b} with $g$ tight $m$-ary, $f$ tight $n$-ary, and $j+1<i\le m$, write
$g=g'\circ_1\theta_{j+1}$; then
\begin{align*}
  g\beta_j\circ_i f &= (g'\circ_1\theta_{j+1})\beta_j\circ_i f \\
                    &= (g'\circ_1\theta_{j+1}\beta_j)\circ_i f   \tag{Proposition~\ref{prop:eq-e}}
  \\
                    &= (g'\circ_{i-j}f)\circ_1\theta_{j+1}\beta_j \\
                    &= ((g'\circ_{i-j}f)\circ_1\theta_{j+1})\beta_j
                      \tag{Proposition~\ref{prop:eq-e}} \\ 
                    &= ((g'\circ_1\theta_{j+1})\circ_i f)\beta_j \\
  &= (g\circ_i f)\beta_{j}. \qedhere
\end{align*}

Next we investigate the various conditions on $s$ in terms of the
corresponding assignments $g\mapsto g\beta_j$. First, we observe:

\begin{proposition}
  If $g$ is tight $m$-ary and $2\le j<j+1<i<m$ then  $g\beta_i\beta_j=g\beta_j\beta_i$.
\end{proposition}

\proof
Write $g=g'\circ_1\theta_{i-1}$. Then
\begin{align*}
  (g\beta_i)\beta_j &= (g'\beta_2\circ_1\theta_{i-1})\beta_j
                      \tag{defn} \\
                    &= g'\beta_2\circ_1 \theta_{i-1}\beta_j \tag{by \eqref{eq:a} with $i=1$} \\                                          &= g'\beta_2\circ_1 (\theta_{i-j+1}\circ_1\theta_{j-1})\beta_j \\
  &= g'\beta_2\circ_1(\theta_{i-j+1}\beta_2\circ_1\theta_{j-1})
    \tag{defn} \\
                    &= (g'\beta_2\circ_1\theta_{i-j+1}\beta_2)\circ_1\theta_{j-1} \\
  &= (g'\circ_1\theta_{i-j+1}\beta_2)\beta_{i-j+2}\circ_1\theta_{j-1}
    \tag{by  \eqref{eq:e}}
  \\
  &= ((g'\circ_1\theta_{i-j+1}\beta_2)\circ_1\theta_{j-1})\beta_i
    \tag{by  \eqref{eq:e}} \\
                    &= (g'\circ_1(\theta_{i-j+1}\beta_2\circ_1\theta_{j-1}))\beta_i \\
  &= (g'\circ_1(\theta_{i-j+1}\circ_1\theta_{j-1})\beta_j)\beta_i
    \tag{defn} \\
                    &= (g'\circ_1\theta_{i-1}\beta_j)\beta_i \\
                    &=  ( (g'\circ_1\theta_{i-1})\beta_j ) \beta_i
                      \tag{by \eqref{eq:a} with $i=1$} \\
  &=  (g\beta_j ) \beta_i. \qedhere
\end{align*}

\begin{proposition}
  Let $s\colon (ab)c\to (ac)b$ be natural isomorphisms in \ca, and $g\mapsto
  g\beta_j$ the corresponding actions on tight morphisms in \bba. The
  following conditions are equivalent:
  \begin{enumerate}[(a)]
  \item $s$ satisfies \eqref{eq:S*};
    \item $\theta_2\circ_2\theta_3\beta_2 =
      (\theta_2\circ_2\theta_3)\beta_3$;
      \item \eqref{eq:a} holds for all tight $f$ and $g$.
  \end{enumerate}
\end{proposition}

\proof
The equivalence of (a) and (b) is a straightforward translation using
universality of the $\theta$s, while (b) is a special case of (c). So
we just need to check that the general case follows from the special
one.

First we show that $f\circ_2\theta_3\beta_2=(f\circ_2\theta_3)\beta_3$
for all tight $f$.  Write $f=f'\circ_1\theta_2$ with $f'$ tight, then:
\begin{align*}
  f\circ_2\theta_3\beta_2 &= (f'\circ_1\theta_2)\circ_2\theta_3\beta_2
  \\
                          &= f'\circ_1(\theta_2\circ_2\theta_3\beta_2) \\
                          &= f'\circ_1(\theta_2\circ_2\theta_3)\beta_3 \tag{by (b)} \\
  &= (f'\circ_1(\theta_2\circ_2\theta_3))\beta_3 \tag{by \eqref{eq:a} with $i=1$} \\
                          &= ((f'\circ_1\theta_2)\circ_2\theta_3)\beta_3 \\
  &= (f\circ_2\theta_3)\beta_3.
\end{align*}
 For $i>2$ write $f=f'\circ_1\theta_{i-1}$, and observe that
\begin{align*}
  f\circ_i \theta_3\beta_2 &=
                             (f'\circ_1\theta_{i-1})\circ_i\theta_3\beta_2
  \\
                           &= (f'\circ_2\theta_3\beta_2)\circ_1\theta_{i-1} \\
                           &= (f'\circ_2\theta_3)\beta_3\circ_1\theta_{i-1} \tag{by the preceding calculation} \\
  &= ((f'\circ_2\theta_3)\circ_1\theta_{i-1})\beta_{i+1} \tag{by  
    \eqref{eq:e}} \\
                           &= ((f'\circ_1\theta_{i-1})\circ_i\theta_3)\beta_{i+1} \\
  &= (f\circ_i\theta_3)\beta_{i+1}.
\end{align*}
 Thus \eqref{eq:a} holds for $g=\theta_3$. For tight $g$ and
$j=2$ write $g=g'\circ_1\theta_3$, so that
\begin{align*}
  f\circ_i g\beta_2 &= f\circ_i(g'\circ_1\theta_3\beta_2) \\
                    &= (f\circ_i g')\circ_i\theta_{3}\beta_2 \\
                    &= ((f\circ_i g')\circ_i\theta_3)\beta_{i+1} \tag{by the preceding calculation} \\
                    &= (f\circ_i(g'\circ_1\theta_3))\beta_{i+1} \\
  &= (f\circ_i g)\beta_{i+1}.
\end{align*}
Finally for $j>2$ write $g=g''\circ_1\theta_{j-1}$, so that
\begin{align*}
  f\circ_ig\beta_j &= f\circ_i (g''\beta_2\circ_1\theta_{j-1}) \\
                   &= (f\circ_i g''\beta_2)\circ_i\theta_{j-1} \\
                   &= (f\circ_i g'')\beta_{i+1}\circ_i\theta_{j-1} \tag{by the preceding calculation} \\
                   &= ((f\circ_i g'')\circ_i\theta_{j-1})\beta_{i+j-1}
                     \tag{by \eqref{eq:e}}\\
                   &= (f\circ_i(g''\circ_1\theta_{j-1}))\beta_{i+j-1} \\
  &= (f\circ_i g)\beta_{i+j-1}
\end{align*}
as required.
\endproof

\begin{proposition}
  Let $s\colon (ab)c\to (ac)b$ be natural isomorphisms in \ca, and $g\mapsto
  g\beta_j$ the corresponding actions on tight morphisms in \bba. The
  following conditions are equivalent:
  \begin{enumerate}[(a)]
  \item $s$ satisfies \eqref{eq:S3a};
    \item $\theta_3\beta_2\circ_3\theta_2 =
     (\theta_3\circ_2\theta_2)\beta_3\beta_2$;
      \item \eqref{eq:c} holds for all tight $f$ and $g$.
  \end{enumerate}
\end{proposition}

\proof
The equivalence of (a) and (b) is a straightforward translation using
universality of the $\theta$s, while (b) is a special case of (c). So
we just need to check that the general case follows from the special
one.

First we show that $g\beta_2\circ_3\theta_2 =
(g\circ_2\theta_2)\beta_3\beta_2$ for all tight $g$ (of arity at least
3). Write $g=g'\circ_1\theta_3$; then
\begin{align*}
  g\beta_2\circ_3\theta_2 &= (g'\circ_1\theta_3)\beta_2\circ_3\theta_2
  \\
                          &= (g'\circ_1\theta_3\beta_2)\circ_3\theta_2
                            \tag{by  \eqref{eq:e}} \\
                          &= g'\circ_1(\theta_3\beta_2\circ_3\theta_2) \\
                          &= g'\circ_1(\theta_3\circ_2\theta_2)\beta_3\beta_2 \tag{by (b)} \\
  &= (g'\circ_1(\theta_3\circ_2\theta_2)\beta_3)\beta_2 \tag{by 
    \eqref{eq:a} with $i=1$} \\
  &= (g'\circ_1(\theta_3\circ_2\theta_2))\beta_3\beta_2 \tag{by 
    \eqref{eq:a} with $i=1$} \\
  &= (g\circ_2\theta_2)\beta_3\beta_2.
\end{align*}
For  $j>2$ write $g=g''\circ_1\theta_{j-1}$, and observe that 
\begin{align*}
  g\beta_j\circ_{j+1}\theta_2 &=
                                (g''\circ_1\theta_{j-1})\beta_j\circ_{j+1}\theta_2  \\
                              &=
                                (g''\beta_2\circ_1\theta_{j-1})\circ_{j+1}\theta_2
                                \tag{by \eqref{eq:e}} \\
                              &= (g''\beta_2\circ_3\theta_{2})\circ_1\theta_{j-1} \\
                              &= (g''\circ_2\theta_2)\beta_3\beta_2\circ_1\theta_{j-1} \tag{by the preceding calculation} \\
  &= ((g''\circ_2\theta_2)\beta_3\circ_1\theta_{j-1})\beta_j
    \tag{by \eqref{eq:e}} \\
  &= ((g''\circ_2\theta_2)\circ_1\theta_{j-1})\beta_{j+1}\beta_j
    \tag{by \eqref{eq:e}} \\ 
  &= (g\circ_j\theta_2)\beta_{j+1}\beta_j. 
\end{align*}
Finally prove the general case by induction on the arity $n$ of
$f$. The base case $n=1$ holds by Proposition~\ref{prop:s-natural}. Write
$f=f'\circ_1\theta_2$; then $f'$ has arity strictly less than that of
$f$, allowing the induction, and 
\begin{align*}
  g\beta_j\circ_{j+1} f &= g\beta_j\circ_{j+1} (f'\circ_1\theta_2) \\
                        &= (g\beta_j\circ_{j+1} f')\circ_{j+1}\theta_2 \\
  &= (g\circ_j f')\beta_{j+n-2}\ldots\beta_j\circ_{j+1}\theta_2
    \tag{ind hyp} \\
  &= ((g\circ_j
    f')\beta_{j+n-2}\ldots\beta_{j+1}\circ_j\theta_2)\beta_{j+1}\beta_j
    \tag{previous case} \\
  &= ((g\circ_j f')\circ_j\theta_2)\beta_{j+n-1}\ldots\beta_j \tag{by \eqref{eq:e}} \\ 
  &= (g\circ_j f)\beta_{j+n-1}\ldots\beta_j
\end{align*}
as required.
\endproof

Similarly we have:

\begin{proposition}
  Let $s\colon (ab)c\to (ac)b$ be natural isomorphisms in \ca, and $g\mapsto
  g\beta_j$ the corresponding actions on tight morphisms in \bba. The
  following conditions are equivalent:
  \begin{enumerate}[(a)]
  \item $s$ satisfies \eqref{eq:S3b};
    \item $\theta_3\beta_2\circ_2\theta_2 =
      (\theta_3\circ_3\theta_2)\beta_2\beta_3$;
      \item \eqref{eq:d} holds for all tight $f$ and $g$. \qedhere
  \end{enumerate}
\end{proposition}

The next result completes the proof of Theorem~\ref{thm:tight-parts}.

\begin{proposition}
  Let $s\colon (ab)c\to (ac)b$ be natural isomorphisms in \ca, and $g\mapsto
  g\beta_j$ the corresponding actions on tight morphisms in
  \bba. If \eqref{eq:a} holds for all tight $g$ and $f$ then the
  following conditions are equivalent:
  \begin{enumerate}[(a)]
  \item $s$ satisfies \eqref{eq:S2};
    \item $\theta_4\beta_2\beta_3\beta_2=\theta_4\beta_3\beta_2\beta_3$;
      \item
        $g\beta_j\beta_{j+1}\beta_j=g\beta_{j+1}\beta_j\beta_{j+1}$
                for all tight $n$-ary $g$. 
  \end{enumerate}
\end{proposition}

\proof
Once again  (a) and (b) are clearly equivalent and (b) is a special case of (c).  Thus it remains to prove that
 (b) implies (c).  First observe that if $g=g'\circ_1\theta_4$ then
\begin{align*}
  g\beta_2\beta_3\beta_2 &= (g'\circ_1\theta_4)\beta_2\beta_3\beta_2 
  \\
                         &= g'\circ_1\theta_4\beta_2\beta_3\beta_2
                           \tag{by \eqref{eq:a}} \\
                         &= g'\circ_1\theta_4\beta_3\beta_2\beta_3
                           \tag{by
    (b)} \\
                         &= (g'\circ_1\theta_4)\beta_3\beta_2\beta_3 \tag{by \eqref{eq:a}} \\
  &= g\beta_3\beta_2\beta_3 
\end{align*}
and now if $g=g''\circ_1\theta_{j-1}$ then
\begin{align*}
  g\beta_j\beta_{j+1}\beta_j &=
                               (g''\circ_1\theta_{j-1})\beta_j\beta_{j+1}\beta_j
  \\
                             &=
                               g''\beta_2\beta_3\beta_2\circ_1\theta_{j-1}
                               \tag{by \eqref{eq:e}} \\
                             &= g''\beta_3\beta_2\beta_3\circ_1\theta_{j-1} \tag{by
    preceding calculation} \\
                             &=
                               (g''\circ_1\theta_{j-1})\beta_{j+1}\beta_j\beta_{j+1}
                               \tag{by \eqref{eq:e}} \\
  &= g\beta_{j+1}\beta_j\beta_{j+1}.
\end{align*}
\endproof

This completes the proof of Theorem~\ref{thm:tight-parts}.
%
%
We now turn to the loose maps.  A loose $m$-ary morphism $g$ can be written as $g'\circ_1 u$ for a unique tight $(m+1)$-ary morphism, where $u$ is the (loose) nullary
morphism classifier.  If $1\le j<m$ we define $g\beta_j= g'\beta_{j+1}\circ_1 u$.
This clearly gives an action of the braid groups on loose
morphisms and moreover this definition is forced by
the requirement that the equation \eqref{eq:e} hold for loose
morphisms, as it must in a braided skew multicategory;
furthermore, the actions on tight and loose maps are mutually
compatible.

\begin{lemma} \label{lemma:f=u} Equations \eqref{eq:b}--\eqref{eq:e} hold for
  all tight $g$ with $f=u$.
\end{lemma}

\begin{proof}
For \eqref{eq:e}, first consider the case that $i=1$. This says
that $g\beta_j\circ_1 u=(g\circ_1 u)\beta_{j-1}$ and holds by
definition of the right hand side.  
For $i >1$ write $g=h\circ_1\theta_i$ with $h$ tight. Then
\begin{align*}
  g\beta_j \circ_i u &= (h\circ_1\theta_i)\beta_j\circ_i u \\ 
                     &= (h\beta_{j-i+1}\circ_1 \theta_i)\circ_i u
                       \tag{by \eqref{eq:e}} \\
  &= h\beta_{j-i+1}\circ_1(\theta_i\circ_i u) \\
                      &= (h\circ_1(\theta_i\circ_i u))\beta_{j-1}
                       \tag{by \eqref{eq:e} and fact that $\theta_i\circ_i u$ is tight} \\
                     &= ((h\circ_1 \theta_i)\circ_i u)\beta_{j-1} \\
  &= (g\circ_i u)\beta_{j -1}.
\end{align*}

For \eqref{eq:c} we first observe that the special case $\theta_{3}\beta_{2} \circ_{3} u = \theta_{3} \circ_{2} u$ is equivalent to the equality in Proposition~\ref{prop:s.r}.  Thus it remains to show that the general case follows from this special case.  For this let us write $g = g' \circ \theta_{j+1}$.
\begin{align*}
  g\beta_j\circ_{j+1} u &= (g' \circ_1 \theta_{j+1})\beta_{j} \circ_{j+1} u \\
                    &=  ((g' \circ_1 \theta_{3}) \circ_1 \theta_{j-1})\beta_{j} \circ_{j+1} u  \\
                     &= ((g' \circ_1 \theta_{3})\beta_2 \circ_1 \theta_{j-1}) \circ_{j+1} u \tag{by \eqref{eq:e}} \\
                     &= ((g' \circ_1 \theta_{3}\beta_2)
                       \circ_1\theta_{j-1}) \circ_{j+1} u \tag{by
                       \eqref{eq:a}} \\
                        &=
                          (g'\circ_1(\theta_3\beta_2\circ_1\theta_{j-1}))\circ_{j+1}
                          u \\
                       &= g'\circ_1((\theta_3\beta_2\circ_1\theta_{j-1})\circ_{j+1}u) \\
                       &= g'\circ_1 ((\theta_3\beta_2\circ_3 u)\circ_1\theta_{j-1}) \\
                       &= g'\circ_1 ((\theta_3\circ_2 u)\circ_1\theta_{j-1}) \tag{by special case}\\
                       &= g'\circ_1 ((\theta_3\circ_1\theta_{j-1})\circ_j u) \\
 &= (g'\circ_1\theta_{j+1})\circ_j u \\
                   &= g \circ_{j} u.
\end{align*}

For \eqref{eq:d} observe that the special case
$\theta_{3}\beta_{2} \circ_{2} u = \theta_{3} \circ_{3} u$ is
equivalent to the equality in Proposition~\ref{prop:s.r1}, and
use a similar argument to show that the general case follows
from this special case.

Finally for \eqref{eq:b} write $g=h\circ_1 \theta_{j+1}$ with $h$
tight. Then 
\begin{align*}
  g\beta_j\circ_i u &= (h\circ_1 \theta_{j+1})\beta_j \circ_i u \\
                    &= (h\circ_1 \theta_{j+1}\beta_j)\circ_i u \tag{by \eqref{eq:a}} \\
                    &= (h\circ_{i-j+1} u)\circ_1 \theta_{j+1}\beta_j \\
  &= ((h\circ_{i-j+1}u)\circ_1 \theta_{j+1})\beta_j \tag{by
    \eqref{eq:a} and fact that $h\circ_{i-j+1}u$ is tight}
  \\
                    &= ((h\circ_1 \theta_{j+1})\circ_i u)\beta_j \\
  &= (g\circ_i u)\beta_j. \qedhere
\end{align*}
\end{proof}

To complete the proof of Theorem~\ref{thm:multicat} it remains to show:

\begin{proposition}
  The remaining conditions \eqref{eq:a}--\eqref{eq:e} also hold when
  $f$ or $g$ are loose. 
\end{proposition}

\proof
Consider \eqref{eq:e}.  If $g$ is tight, but $f$ is loose, write $f=f'\circ_1 u$
with $f'$ tight of arity $n+1$; then
\begin{align*}
  g\beta_j \circ_i f &= g\beta_j \circ_i(f'\circ_1 u) \\
                     &= (g\beta_j\circ_i f')\circ_i u \\
                     &= (g\circ_i f')\beta_{j+n}\circ_i u
                       \tag{\text{\eqref{eq:e}} for tight morphisms} \\
                     &= ((g\circ_i f')\circ_i u) \beta_{j+n-1}
                       \tag{by Lemma~\ref{lemma:f=u}} 
                   \\
                     &= (g\circ_i (f'\circ_1 u))\beta_{j+n-1} \\
  &= (g\circ_i f)\beta_{j+n-1} .
\end{align*}
If $g$ is loose, write $g=g'\circ_1 u$
with $g'$ tight, and then observe that
\begin{align*} 
  g\beta_j \circ_1 u &= (g'\circ_1 u)\beta_j \circ_1 u \\
                     &= (g'\beta_{j+1}\circ_1 u)\circ_1 u \tag{defn} \\
                     &= (g'\beta_{j+1}\circ_2 u)\circ_1 u \\
                     &= (g'\circ_2 u)\beta_j\circ_1 u \tag{previous case} \\
                     &= ((g'\circ_2u)\circ_1 u)\beta_{j-1} \tag{defn} \\
                     &= ((g'\circ_1 u)\circ_1 u)\beta_{j-1} \\
  &= (g\circ_1 u)\beta_{j-1} \\
  g\beta_j \circ_i f &= (g'\circ_1u)\beta_j\circ_i f \\
                     &= (g'\beta_{j+1}\circ_1 u)\circ_i f
                        \tag{defn} \\
                     &= ((g'\beta_{j+1}\circ_{i+1} f)\circ_1 u \\
                     &= ((g'\circ_{i+1} f)\beta_{j+n}\circ_1 u
                       \tag{previous case} \\
                     &= ((g'\circ_{i+1}f)\circ_1u)\beta_{j+n-1}
                        \tag{previous case} \\
                     &= ((g'\circ_1 u)\circ_i f)\beta_{j+n-1} \\
  &= (g\circ_i f)\beta_{j+n-1} 
\end{align*}
which completes all cases of \eqref{eq:e}.

Next we do \eqref{eq:a}. If $f$ is tight but $g$ loose, write
$g=g'\circ_1 u$ with $g'$ tight. Then
\begin{align*}
  f\circ_i g\beta_j &= f\circ_i (g'\circ_1 u)\beta_j \\
                    &= f\circ_i (g'\beta_{j+1}\circ_1u)  \tag{defn} \\
                    &= (f\circ_ig'\beta_{j+1})\circ_iu \\
                    &= (f\circ_i g')\beta_{i+j}\circ_i u \tag{by
  \eqref{eq:e}} \\
                    &= ((f\circ_i g')\circ_iu)\beta_{i+j-1}  \tag{by
                      \eqref{eq:e}}\\
  &= (f\circ_i g)\beta_{i+j-1}.
\end{align*}
If $f$ is loose, write $f=f'\circ_1
u$ with $f'$ tight. Then
\begin{align*}
  f\circ_i g\beta_j &=((f'\circ_1u)\circ_i g\beta_j \\
                    &= (f'\circ_{i+1}g\beta_j)\circ_1 u \\
                    &= (f'\circ_{i+1} g)\beta_{i+j}\circ_1 u
                      \tag{previous case} \\
                    &= ((f'\circ_{i+1} g)\circ_1 u)\beta_{i+j-1}
                      \tag{by \eqref{eq:e}} \\
  &= (f\circ_i g)\beta_{i+j-1}
\end{align*}
which completes all cases of \eqref{eq:a}. 

Now consider \eqref{eq:b}.   If $g$ is tight but $f$ loose, write $f=f'\circ_1 u$ with $f'$ tight;
then 
\begin{align*}
  g\beta_j\circ_i f &= g\beta_j \circ_i (f'\circ_1 u) \\
                    &= (g\beta_j\circ_i f')\circ_i u \\
                    &= (g\circ_i f')\beta_j\circ_i u  \tag{\text{by
                      \eqref{eq:b} for tight morphisms}} \\
                    &= ((g\circ_i f')\circ_i u)\beta_j \tag{by Lemma~\ref{lemma:f=u}} \\
  &= (g\circ_i f)\beta_j.
\end{align*}
If $g$ is loose write $g=g'\circ_1 u$ with $g'$ tight:
\begin{align*}
  g\beta_j\circ_i f &= (g'\circ_1 u)\beta_j\circ_i f  \\
                    &= (g'\beta_{j+1}\circ_1 u)\circ_i f \tag{by Lemma~\ref{lemma:f=u}} 
                     \\
                    &= (g'\beta_{j+1}\circ_{i+1}f)\circ_1 u \\
                    &= (g'\circ_{i+1}f)\beta_{j+1}\circ_1 u
                      \tag{\text{previous case}} \\
                    &= ((g'\circ_{i+1}f)\circ_1 u)\beta_{j} \tag{by \eqref{eq:e}} \\
  &= (g\circ_if)\beta_j.
\end{align*}
This completes \eqref{eq:b}. 

Next we do \eqref{eq:c}, so that $j=i-1$. If $g$ is tight and $f$ loose, write
$f=f'\circ_1 u$ with $f'$ tight; then:
\begin{align*}
  g\beta_j\circ_{j+1} f &= g\beta_j \circ_{j+1} (f'\circ_1 u) \\
                    &= (g\beta_j\circ_{j+1} f')\circ_{j+1} u \\
                    &= (g\circ_{j}
                      f')\beta_{j+n}\ldots\beta_j\circ_{j+1} u \tag{by \eqref{eq:c} for tight morphisms} \\
                    &=
                      (g\circ_{i-1}f')\beta_{j+n}\ldots\beta_{j+1}\circ_{j}u
                      \tag{by Lemma~\ref{lemma:f=u}}  \\
                    &= ((g\circ_{i-1}f')\circ_{i-1}u)
                      \beta_{j+n-1}\ldots\beta_j  \tag{by
                      repeated use of \eqref{eq:e}} \\
  &= (g\circ_{i-1}f)\beta_{j+n-1}\ldots\beta_j
\end{align*} 
while if $g$ is loose we may write $g=g'\circ_1 u$ and 
\begin{align*}
  g\beta_j\circ_{j+1} f &= (g'\circ_1 u)\beta_j\circ_{j+1} f  \\
                    &= (g'\beta_{j+1}\circ_1 u)\circ_{j+1} f
                      \tag{defn}
                    \\
                    &= (g'\beta_{j+1}\circ_{j+2}f)\circ_1 u \\
                    &=
                      (g'\circ_{i}f)\beta_{j+n}\ldots\beta_{j+1}\circ_1
                      u \tag{previous case} \\
                    &= ((g'\circ_{i}f)\circ_1
                      u)\beta_{j+n-1}\ldots\beta_j \tag{by
                      repeated use of \eqref{eq:e}} \\
  &= (g\circ_{i-1}f)\beta_{j+n-1}\ldots\beta_j
\end{align*}
completing the proof of \eqref{eq:c}. 

That leaves only \eqref{eq:d}, where $j=i$. This is entirely
analogous to the case of \eqref{eq:c}, and is left to the
reader. 
\endproof

\bibliographystyle{plain}

\end{document}